\documentclass{amsart}

\usepackage{amsmath,amssymb}
\usepackage{amsfonts}
\usepackage{graphicx}
\usepackage{latexsym}
\usepackage{array}
\usepackage{subfig}
\usepackage{stmaryrd}
\usepackage{color}
\usepackage{multicol}
\usepackage{vwcol}

\newcommand{\balpha}{{\boldsymbol{\alpha}}}
\newcommand{\bbeta}{{\boldsymbol{\beta}}}
\newcommand{\be}{{\boldsymbol{e}}}

\newtheorem{theorem}{Theorem}[section]
\newtheorem{lemma}[theorem]{Lemma}

\theoremstyle{definition}

\theoremstyle{remark}
\newtheorem{remark}[theorem]{Remark}
\newtheorem{assumption}[theorem]{Assumption}
\numberwithin{equation}{section}

\begin{document}

\title[$H^m$ Nonconforming FEM Spaces When $m=n+1$]{Nonconforming
  Finite Element Spaces for \\ $2m$-th Order Partial Differential
    Equations on $\mathbb{R}^n$ Simplicial Grids when $m=n+1$}
\thanks{
This paper is dedicated to the memory of Professor Ming Wang.  This
study was partially supported by DOE Grant DE-SC0009249 as part of the
Collaboratory on Mathematics for Mesoscopic Modeling of Materials and
by the NSF grant DMS-1522615.}



\author{Shuonan Wu}
\address{Department of Mathematics, Pennsylvania State University,
  University Park, PA 16802, USA}
\curraddr{}
\email{sxw58@psu.edu}

\author{Jinchao Xu}
\address{Department of Mathematics, Pennsylvania State University,
  University Park, PA 16802, USA}
\curraddr{}
\email{xu@math.psu.edu}


\date{}


\begin{abstract}
In this paper, we propose a family of nonconforming finite elements
for $2m$-th order partial differential equations in $\mathbb{R}^n$ on
simplicial grids when $m=n+1$. This family of nonconforming elements
naturally extends the elements proposed by Wang and Xu [Math.
Comp. 82(2013), pp. 25--43] , where $m \leq n$ is required. We prove
the unisolvent property by induction on the dimensions using the
similarity properties of both shape function spaces and degrees of
freedom.  The proposed elements have approximability, pass the
generalized patch test and hence converge.  We also establish quasi-optimal
error estimates in the broken $H^3$ norm for the 2D nonconforming
element. In addition, we propose an $H^3$ nonconforming finite element
that is robust for the sixth order singularly perturbed problems in
2D. These theoretical results are further validated by the numerical
tests for the 2D tri-harmonic problem.
\end{abstract}

\maketitle

\section{Introduction} \label{sec:intro}

In \cite{wang2013minimal}, Wang and Xu proposed a family of
nonconforming finite elements for $2m$-th order elliptic partial
differential equations in $\mathbb{R}^n$ on simplicial grids, with the
requirement that $m \leq n$. These elements (named Morley-Wang-Xu or
MWX elements) are simple and elegant when compared to the conforming
finite elements, with a combination of simplicial geometry, polynomial
space, and convergence analysis. For example, in 3D, the minimal
polynomial degrees of the $H^2$ and $H^3$ conforming finite elements
are $9$ and $17$, respectively (cf.  \cite{alfeld1991structure,
zhang2009family}), while those of MWX elements are only $2$ and
$3$, respectively. In consideration of the desired properties of the
MWX elements, can we extend the MWX elements to the case in which $m >
n$?  In this paper, we partially answer this question by constructing
a family of nonconforming finite elements when $m=n+1$. 

Conforming finite element spaces for $2m$-th order partial
differential equations in $\mathbb{R}^n$ would require $C^{m-1}$
continuity, which could lead to an extremely complicated construction
when $m \geq 2$ or $n \geq 2$. In 2D, the minimal degree of conforming
finite element is $5$, which refers to the well-known Argyris elements (cf.
\cite{ciarlet1978finite}). In \cite{vzenivsek1970interpolation},
{\v{Z}}en{\'\i}{\v{s}}ek constructed the $H^3$ conforming finite
element on the 2D triangular grids.  The construction was further studied
in \cite{bramble1970triangular} for the $H^m$ conforming finite
elements for arbitrary $m\geq 1$ in 2D, which requires a polynomial
of degree $4m-3$. Moreover, the construction and implementation of
conforming finite elements are increasingly daunting with the growth
of dimension $n$.  In fact, the conforming finite elements in 3D, as
far as the authors are aware, have only been implemented when $m \leq
2$ (cf.  \cite{zhang2009family}). An alternative is the conforming
finite elements on rectangular grids for arbitrary $m$ and $n$ (see
\cite{hu2015minimal}).

For the construction of nonconforming finite elements, to remove the
restriction $m \leq n$ is also a daunting
challenge.  In \cite{hu2016canonical}, Hu and Zhang used the full
$\mathcal{P}_{2m-3}$ polynomial space to construct $H^m$
nonconforming finite elements in 2D when $m \geq 4$. For $m=3$ and
$n=2$, they applied the full $\mathcal{P}_4$ polynomial space. In this
paper, we present a universal construction for $H^m$ nonconforming
finite elements when $m=n+1$. The shape function space in this family,
denoted as $P_T^{(n+1,n)}$ on simplex $T$, is defined by
the $\mathcal{P}_{n+1}$ polynomial space enriched by the
$\mathcal{P}_{n+2}$ volume bubble function. With carefully designed
degrees of freedom, we prove the unisolvent property by the
similarities of both shape function spaces and degrees of freedom (see
Lemma \ref{lem:similarity-space}).  Note that for $m=3$ and $n=2$, the
number of local degrees of freedom is $12$ in our element, which is
three less than the element given by Hu and Zhang in
\cite{hu2016canonical}.

The shape function space of $H^3$ nonconforming element in 2D is the
same as the second type of nonconforming element $\tilde{W}_h(T)$
proposed in \cite{nilssen2001robust}, where the authors focused on the
construction of robust nonconforming elements for singularly perturbed
fourth order problems.  The set of degrees of freedom of the proposed
2D element, however, is different from that of $\tilde{W}_h(T)$.  The
extensions of the robust $H^2$ nonconforming elements included
\cite{tai2006discrete} for 3D low-order case, and
\cite{guzman2012family} for the arbitrary polynomial degree.  The
proposed $H^3$ nonconforming finite element space in 2D is $H^1$
conforming and thus is suitable for second order elliptic problems.
Further, by adding additional bubble functions to the shape function
space, a modified $H^3$ nonconforming element can handle both second
and fourth order elliptic problems and thus is robust for the sixth
order singularly perturbed problems in 2D (see Section
\ref{sec:perturbed-2D}).

While the construction presented in this paper is mainly motivated by
theoretical considerations, the new family of elements can also be
applied to several practical problems. For instance, the nonconforming
finite element when $m=3$ and $n=2$ can be applied to many
mathematical models, including the thin-film equations (cf.
\cite{barrett2004finite}), the phase field crystal model (cf.
\cite{backofen2007nucleation, cheng2008efficient, wise2009energy,
hu2009stable, wang2011energy}), the Willmore flows (cf.
\cite{du2004phase, du2005phase}) and the functionalized Cahn-Hilliard
model (cf. \cite{dai2013geometric, doelman2014meander}).

In addition to conforming and nonconforming finite element methods,
the other types of discretization methods for $2m$-th order partial
differential equations may also be feasible. In
\cite{gudi2011interior}, Gudi and Neilan proposed a $C^0$-IPDG method
and a $C^1$-IPDG method for the sixth order elliptic equations on the
2D polygonal domain. These methods, in the framework of discontinuous
Galerkin methods, can be easily implemented, while the
discrete variational forms need to be carefully designed by
introducing certain penalty terms on the element interfaces. Further,
even though the DG methods may not as constrained by
matching dimension to the order of equation, the complexity of the
penalty terms should also be studied with the growth of dimension $n$.
A family of $\mathcal{P}_m$ interior nonconforming finite
element methods in $\mathbb{R}^n$ was proposed in \cite{wu2017pm}
aiming to balance the weak continuity and the complexity of the
penalty terms. Mixed methods are also feasible for high order elliptic
equations, see \cite{li1999full, li2006optimal} for fourth order
equations, \cite{droniou2017mixed} for sixth order equations, and
\cite{schedensack2016new} for 2D $m$th-Laplace equations based on the
Helmholtz decompositions for tensor-valued functions.

The rest of the paper is organized as follows. In Section
\ref{sec:spaces}, we provide a detailed description of our family of
$H^m$ nonconforming finite elements when $m=n+1$.  In Section
\ref{sec:error-estimate}, we state and prove the convergence of the
proposed nonconforming finite elements.  Further, we show a quasi-optimal
error estimate under the conforming relatives assumption.
In Section \ref{sec:perturbed-2D}, we propose an $H^3$
nonconforming finite element that is robust for the sixth order
singularly perturbed problems in 2D. Numerical tests are provided in
Section \ref{sec:tests} to support the theoretical findings, and some
concluding remarks are given in Section \ref{sec:concluding}.  

\section{Nonconforming Finite Element Spaces} \label{sec:spaces}
In this section, we shall construct universal nonconforming finite
elements of $H^{m}(\Omega)$ for $\Omega \subset \mathbb{R}^n$ with
$m=n+1, n \geq 1$. Here, we assume that $\Omega$ is a bounded
polyhedron domain of $\mathbb{R}^n$.


Throughout this paper, we use the standard notation for the usual
Sobolev spaces as in \cite{ciarlet1978finite,
brenner2007mathematical}.  For an $n$-dimensional multi-index $\balpha
= (\alpha_1, \cdots, \alpha_n)$, we define 
$$ 
|\balpha| := \sum_{i=1}^n \alpha_i, \quad 
\partial^\balpha := \frac{\partial^{|\balpha|}}{\partial x_1^{\alpha_1}
\cdots \partial x_n^{\alpha_n}}.
$$ 
Given an integer $k\geq 0$ and a bounded domain $G \subset
\mathbb{R}^n$ with boundary $\partial G$, let $H^k(G), H_0^k(G),
\|\cdot\|_{k,G}$, and $|\cdot|_{k,G}$ denote the usual Sobolev spaces,
norm, and semi-norm, respectively. 

Let $\mathcal{T}_h$ be a conforming and shape-regular simplicial
triangulation of $\Omega$ and $\mathcal{F}_h$ be the set of all
faces of $\mathcal{T}_h$. Let $\mathcal{F}_h^i := \mathcal{F}_h
\setminus \partial \Omega$ and $\mathcal{F}_h^\partial := \mathcal{F}_h
\cap \partial\Omega$.  Here, $h := \max_{T\in \mathcal{T}_h}$,
and $h_T$ is the diameter of $T$ (cf. \cite{ciarlet1978finite,
    brenner2007mathematical}).  We assume that $\mathcal{T}_h$ is
quasi-uniform, namely 
$$    
\exists \eta>0 \mbox{~~such that~~} \max_{T\in \mathcal{T}_h}
\frac{h}{h_T} \leq \eta,
$$
where $\eta$ is a constant independent of $h$. Based on the
triangulation $\mathcal{T}_h$, for $v \in L^2(\Omega)$ with
$v|_T \in H^k(T), \forall T\in \mathcal{T}_h$, we define
$\partial^{\balpha}_h v$ as the piecewise partial derivatives of $v$
when $|\balpha| \leq k$, and
$$ 
\|v\|_{k,h}^2 := \sum_{T \in \mathcal{T}_h}\|v\|_{k,T}^2, \quad
|v|_{k,h}^2 := \sum_{T\in \mathcal{T}_h} |v|_{k,T}^2.
$$ 

For convenience, we use $C$ to denote a generic positive
constant that may stand for different values at its different
occurrences but is independent of the mesh size $h$. The notation $X
\lesssim Y$ means $X \leq CY$.

\subsection{The $H^m$ nonconforming finite elements when $m=n+1$}

Following the description of \cite{ciarlet1978finite,
brenner2007mathematical}, a finite element can be represented by a
triple $(T, P_T, D_T)$, where $T$ is the geometric shape of the
element, $P_T$ is the shape function space, and $D_T$ is the set of
the degrees of freedom that is $P_T$-unisolvent. 

Let $T$ be an $n$-simplex.  Given an $n$-simplex $T$ with
vertices $a_i, 1\leq i \leq n+1$, let
$\lambda_1, \lambda_2, \cdots, \lambda_{n+1}$ be the barycenter
coordinates of $T$. For $1\leq k \leq n$, let $\mathcal{F}_{T,k}$ be
the set consisting of all $(n-k)$-dimension sub-simplexes of $T$.  For
any $F \in \mathcal{F}_{T,k}$, let $|F|$ denote its
$(n-k)$-dimensional measure, and $\nu_{F,1}, \cdots, \nu_{F,k}$ be
linearly independent unit vectors that are orthogonal to the tangent
space of $F$. Specifically, $F$ represents a vertex and $|F| = 1$ when
$k=n$. 

For any simplex $K$, let $q_K$ be the volume bubble function of the
simplex $K$. Specifically, we have  
$$ 
q_T = \lambda_1\lambda_2 \cdots \lambda_{n+1}.
$$ 
The shape function space $P_T = P_T^{(m,n)}$ when $m=n+1$ is defined
as 
\begin{equation} \label{equ:shape-function} 
P_T^{(n+1,n)} := \mathcal{P}_{n+1}(T) + q_T\mathcal{P}_1(T),
\end{equation}
where $\mathcal{P}_k(T)$ denotes the space of all polynomials defined
on $T$ with a degree not greater than $k$, for any integer $k \geq 0$.

For $k\geq 1$, let $A_k$ be the set consisting of all multi-indexes
$\balpha$ with $\sum_{i=k+1}^n \alpha_i = 0$.  For $1\leq k \leq n$,
any $(n-k)$-dimensional sub-simplex $F\in \mathcal{F}_{T,k}$ and
$\balpha\in A_k$ with $|\balpha|=n+1-k$, we define 
\begin{equation} \label{equ:DOF-1}
d_{T,F,\balpha}(v) := \frac{1}{|F|}\int_F
\frac{\partial^{n+1-k}v}{\partial\nu_{F,1}^{\alpha_1}
\cdots \partial\nu_{F,k}^{\alpha_k}} \qquad \forall v\in H^{n+1}(\Omega).
\end{equation}
When $|\balpha|=0$, we define 
\begin{equation} \label{equ:DOF-2}
d_{T,a_i,0}(v) := v(a_i) \qquad \forall v \in H^{n+1}(\Omega).
\end{equation}
By the Sobolev embedding theorem (cf. \cite{adams2003sobolev}),
$d_{T,F,\balpha}$ and $d_{T,a_i,0}$ are continuous linear functionals
on $H^{n+1}(T)$. Then, the set of the degrees of freedom is 
\begin{equation} \label{equ:DOF}
\begin{aligned}
D_T^{(n+1,n)} :=~ & \{d_{T,F,\balpha}~:~ \balpha\in A_k \mbox{~with~}
|\balpha|=n+1-k, F\in \mathcal{F}_{T,k}, 1\leq k \leq n\} \\
         \cup~& \{d_{T,a_i,0}~:~ 1 \leq i \leq n+1\}.
\end{aligned}
\end{equation}
We also number the local degrees of freedom by 
$$ 
d_{T,1}, d_{T,2}, \cdots, d_{T,J},
$$ 
where $J$ is the number of local degrees of freedom.

As a natural extension of MWX elements proposed in
\cite{wang2013minimal}, the diagrams of the finite elements for the
case in which $m \leq n+1$ are plotted in Table \ref{tab:m=n+1}.

\begin{table}[!htbp]
\caption{$m \leq n+1$: diagrams of the finite elements.}
\begin{center}
\begin{tabular}{>{\centering\arraybackslash}m{.6cm} |
>{\centering\arraybackslash}m{2.8cm} |
>{\centering\arraybackslash}m{2.8cm} |
>{\centering\arraybackslash}m{2.8cm} @{}m{0pt}@{} }
\hline
$m \backslash n$ & 1 & 2 & 3 \\ \hline
1 &
\includegraphics[width=1.1in]{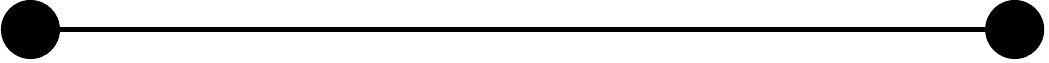}
& 
\includegraphics[width=1.1in]{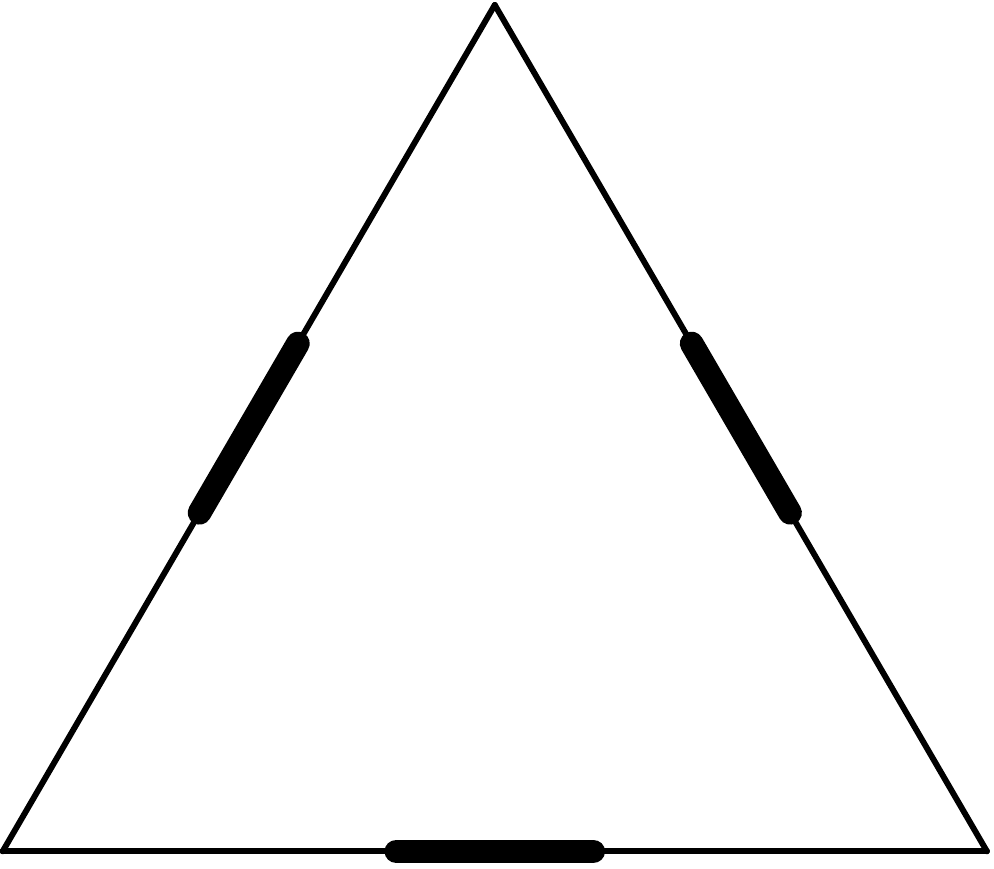}
& 
\includegraphics[width=1.1in]{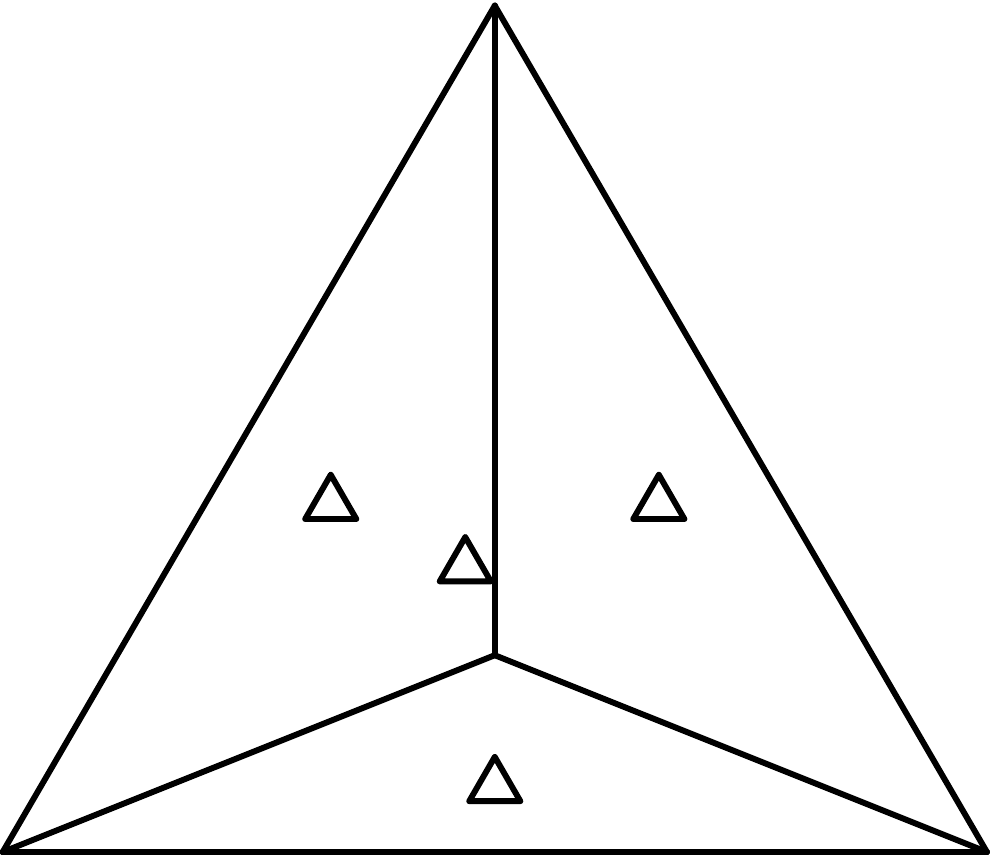} \\ 
\hline 
2 & 
\includegraphics[width=1.1in]{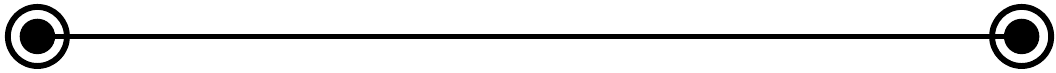}
& 
\includegraphics[width=1.1in]{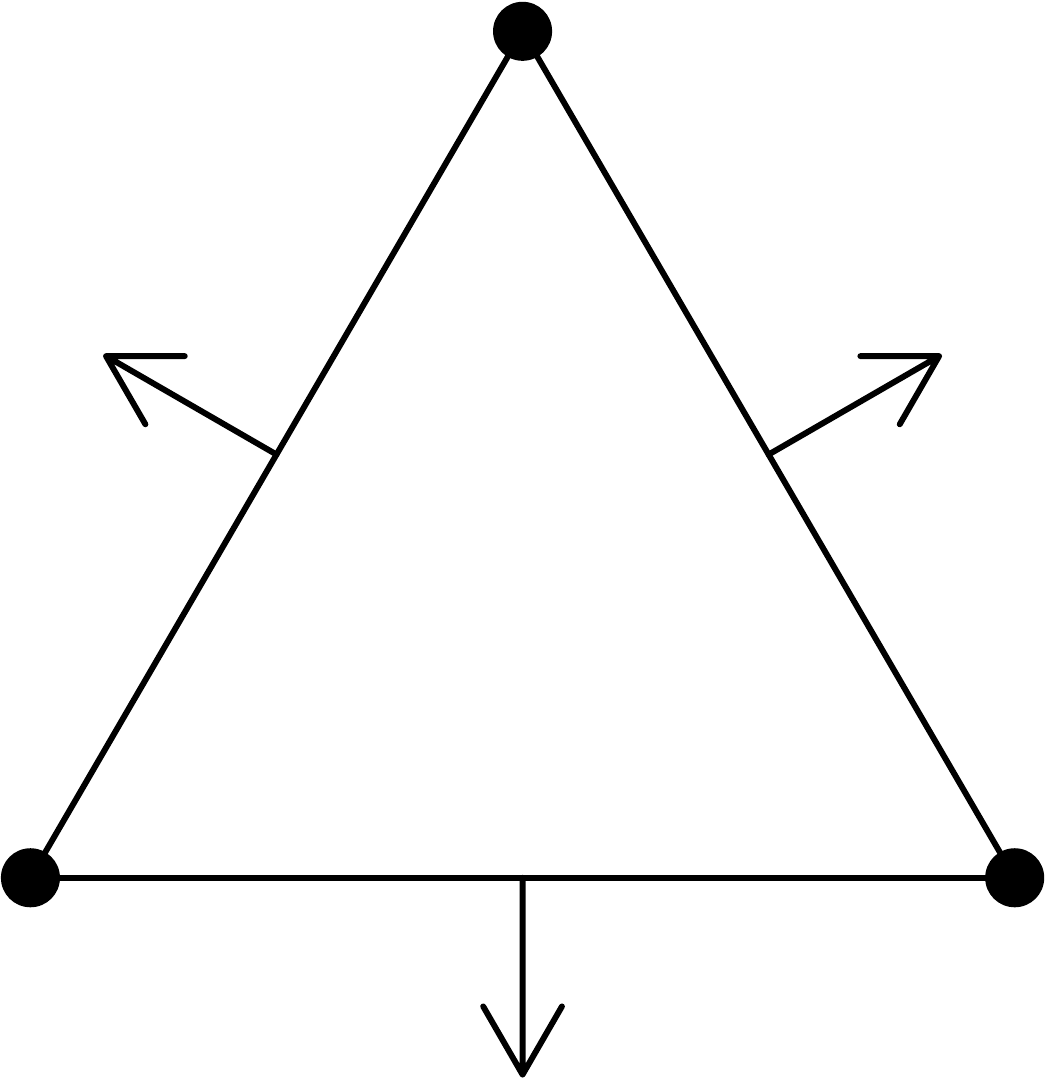}
& 
\includegraphics[width=1.1in]{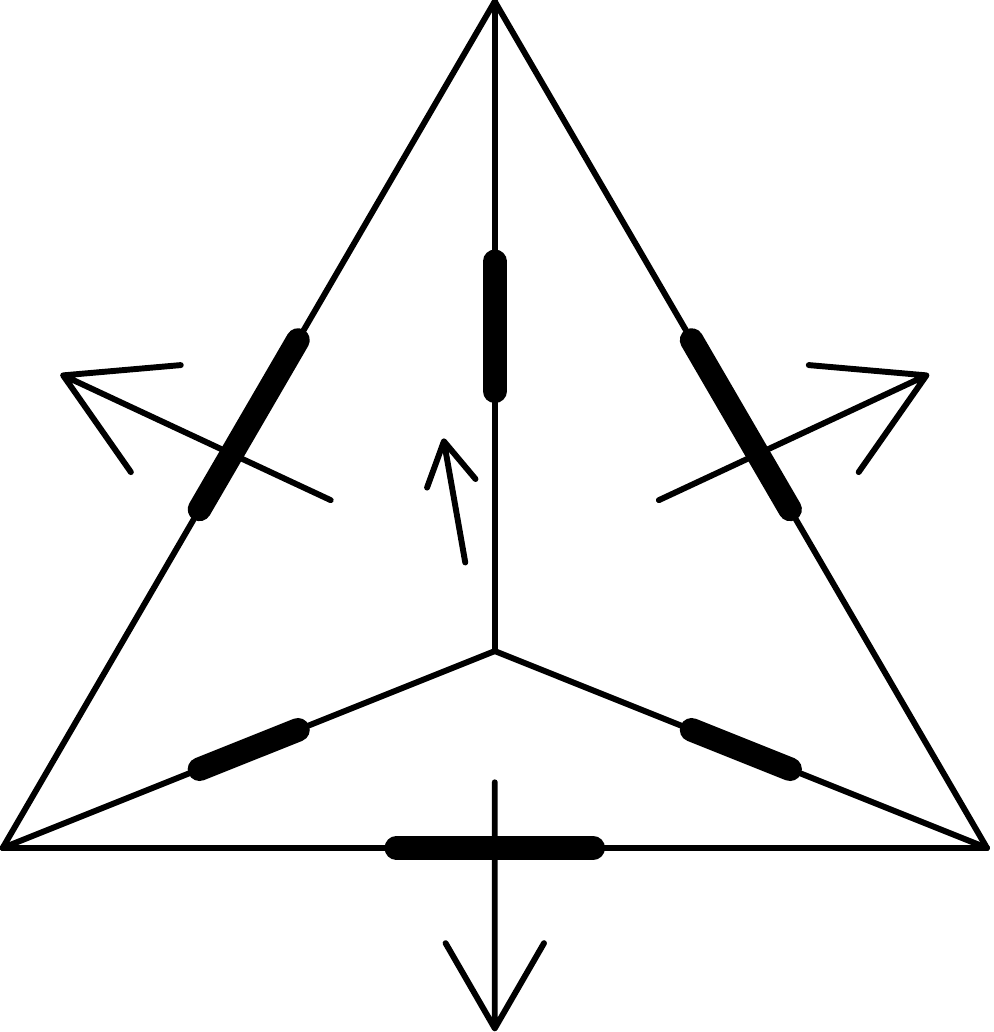}
\\ \hline 
3 & 
&
\includegraphics[width=1.1in]{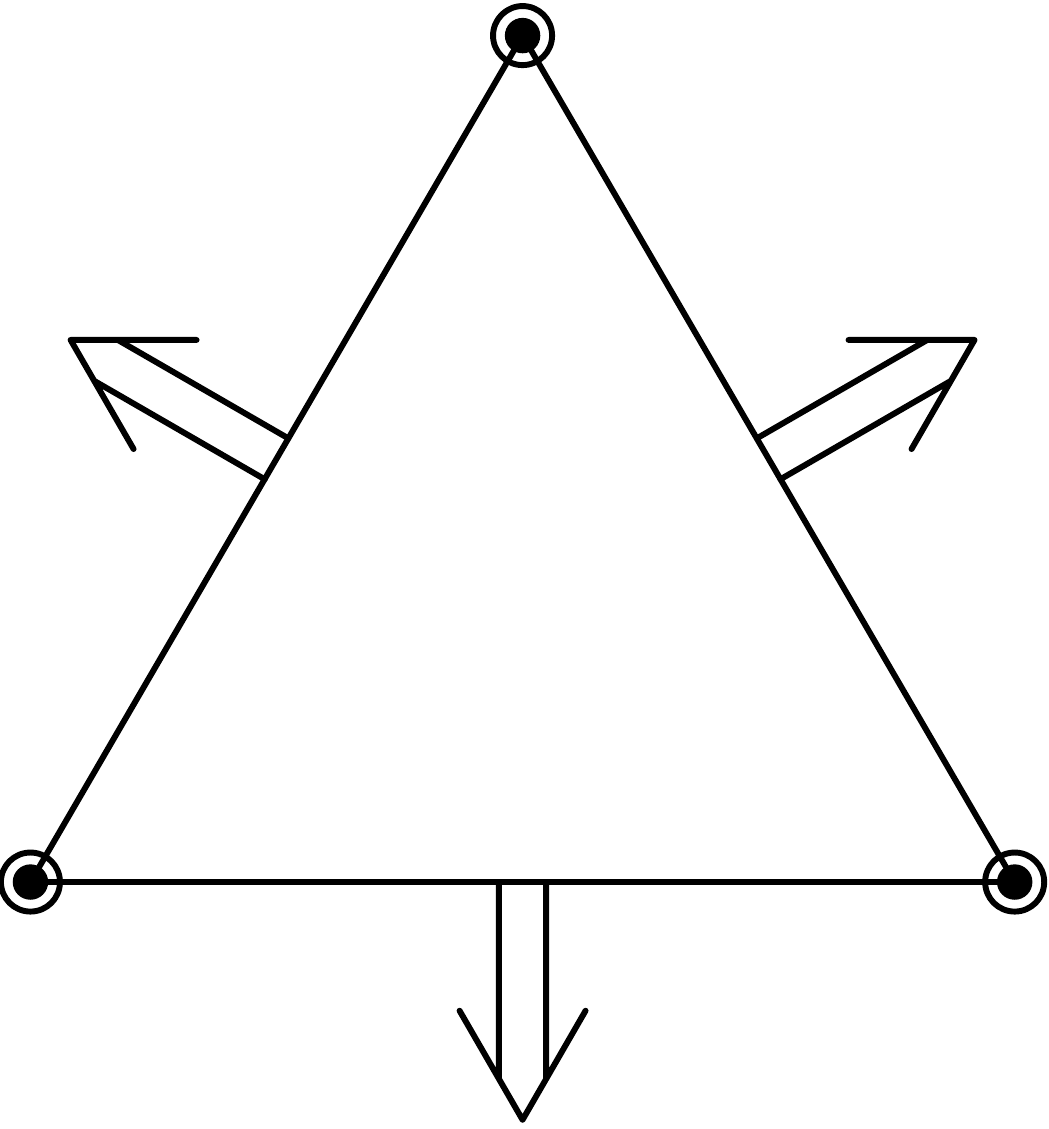}
&
\includegraphics[width=1.1in]{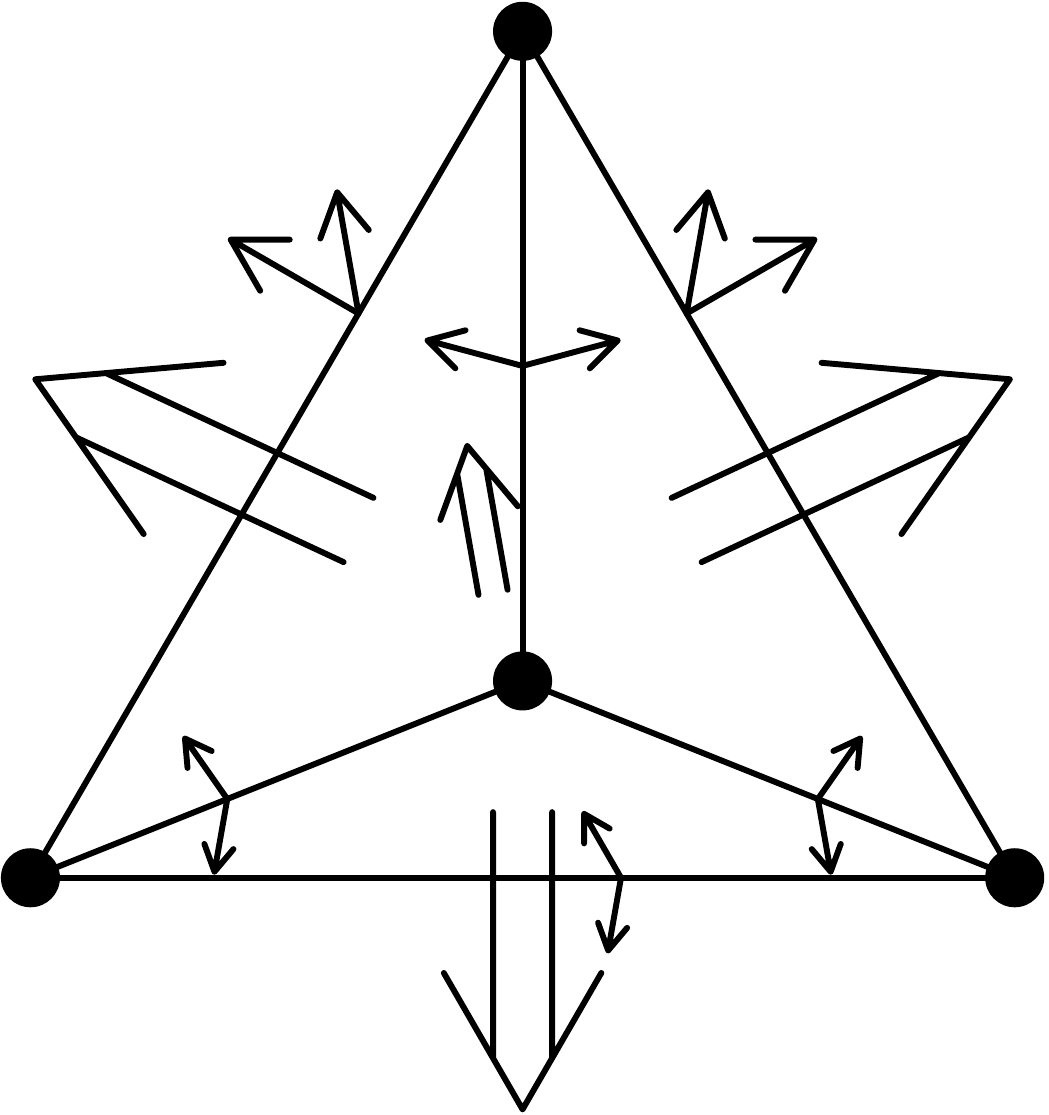}
\\ \hline
4 & 
&
&
\includegraphics[width=1.1in]{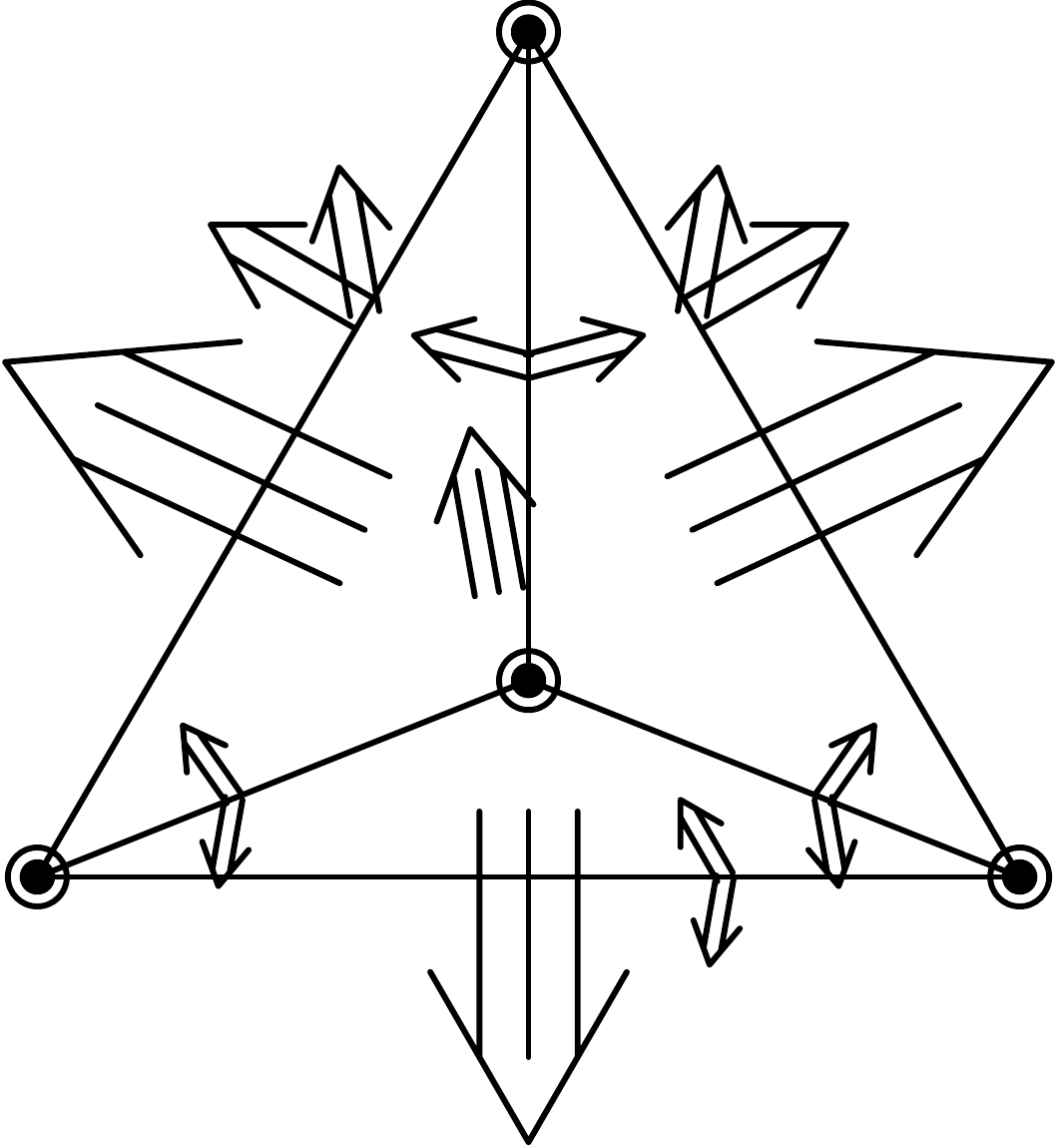}
\\ \hline 
\end{tabular}
\label{tab:m=n+1}
\end{center}
\end{table}

By the Vandermonde combinatorial identity, the number of local degrees
of freedom defined in \eqref{equ:DOF-1} is 
$$ 
\sum_{k=1}^n C_{n+1}^{n-k+1}C_{n}^{n+1-k} = C_{2n+1}^n - 1,
$$ 
where the combinatorial number $C_j^i = \frac{j!}{i!(j-i)!}$ for $j
\geq i$ and $C_j^i = 0$ for $j < i$. Therefore, the number of local 
degrees of freedom defined in \eqref{equ:DOF} is 
\begin{equation} \label{equ:num-DOF}
J = (C_{2n+1}^n - 1) + (n+1) = C_{2n+1}^n + n.
\end{equation}

On the other hand, it is straightforward that 
$$ 
\mathcal{P}_{n+1}(T) \cap q_T\mathcal{P}_1(T) = \mathrm{span} \{q_T\}.
$$ 
Hence, the dimension of $P_T^{(n+1,n)}$ defined in
\eqref{equ:shape-function} is given by 
\begin{equation} \label{equ:dim-PT}
\begin{aligned}
\dim \mathcal{P}_T^{(n+1,n)} &= \dim\mathcal{P}_{n+1}(T) + \dim
(q_T\mathcal{P}_1(T)) - \dim(\mathcal{P}_{n+1}(T)\cap
    q_T\mathcal{P}_1(T)) \\
&= C_{2n+1}^n + n,
\end{aligned}
\end{equation}
which is exactly the number of local degrees of freedom calculated in
\eqref{equ:num-DOF}.

\subsection{Unisolvent property of the new nonconforming finite
elements}  In this section, we shall present a proof for the
unisolvent property of the proposed nonconforming finite elements.
This technique can be applied to the all the cases in which $m \leq
n+1$, while only the $m=n+1$ case is presented for simplicity. 

\begin{lemma} \label{lem:vanish-derivatives}
If all the degrees of freedom defined in \eqref{equ:DOF-1} vanish,
then for $0\leq k \leq n$, any $(n-k)$-dimensional sub-simplex $F\in
\mathcal{F}_{T,k}$, we have  
\begin{equation} \label{equ:m-derivatives} 
\frac{1}{|F|}\int_F \nabla^{n+1-k} v = 0,
\end{equation} 
where $\nabla^{l}$ is the $l$-th Hessian tensor for any
integer $l\geq 0$.
\end{lemma}
\begin{proof}
This lemma can be proved by applying Green's lemma recursively. We
refer to a similar proof in Lemma 2.1 of \cite{wang2013minimal}
for $m \leq n$. 
\end{proof}

For the unisolvent property, we first show the following crucial
lemma. 
\begin{lemma} \label{lem:similarity-space}
If $v \in P_T^{(n+1,n)} =\mathcal{P}_{n+1}(T) + q_T\mathcal{P}_1(T)$
with all the degrees of freedom in \eqref{equ:DOF-1} zero, then 
\begin{enumerate}
\item For $1\leq k \leq n$, any $(n-k)$-dimensional sub-simplex $F \in
\mathcal{F}_{T,k}$,
\begin{equation} \label{equ:similarity-subsimplex}
v|_{F} \in P_{F}^{(n+1-k,n-k)} = \mathcal{P}_{n+1-k}(F) + q_F
\mathcal{P}_1(F).
\end{equation}
\item In particular, for any $(n-1)$-dimensional sub-simplex $F_l \in
\mathcal{F}_{T,1}$, 
\begin{equation} \label{equ:similarity-space}
v|_{F_l} \in P_{F_l}^{(n,n-1)} = \mathcal{P}_{n}(F_l) +
q_{F_l}\mathcal{P}_1(F_l)
\qquad \forall 1 \leq l \leq n+1.
\end{equation}
\end{enumerate}
\end{lemma}
\begin{proof}
(1) is an immediate consequence of (2) by induction.
Without loss of generality, we prove (2) for the case in which $l =
n$.  Applying $\{1, \lambda_1, \cdots, \lambda_n\}$ as the basis of
$\mathcal{P}_1(T)$, then $v$ can be written as 
$$ 
\begin{aligned}
v &= \tilde{v}_{n} + \sum_{i_1 + \cdots + i_n = n+1}
c_{i_1,\cdots,i_n} \lambda_1^{i_1}\cdots \lambda_n^{i_n} +
\sum_{j=1}^n \theta_j \lambda_j q_T \\
&= \tilde{u}_{n} + \lambda_n \hat{u}_n + \sum_{i_1 + \cdots +
    i_{n-1} = n+1} c_{i_1, \cdots, i_{n-1}, 0} \lambda_1^{i_1}
    \cdots \lambda_{n-1}^{i_{n-1}} + \sum_{j=1}^n \theta_j \lambda_j q_T, 
\end{aligned}
$$ 
where $\tilde{u}_n, \hat{u}_n \in \mathcal{P}_n(T)$. Since the volume
average of $(n+1)$-th total derivatives vanishes as shown in Lemma
\ref{lem:vanish-derivatives}, or 
\begin{equation} \label{equ:m-derivatives2}
\frac{1}{|T|}\int_T \frac{\partial^{n+1}v}{\partial \lambda_1^{i_1}
  \cdots \partial \lambda_{n-1}^{i_{n-1}}}  = 0 \qquad \forall i_1 +
  \cdots i_{n-1} = n+1, 
\end{equation}
and 
$$ 
\lambda_jq_T = \lambda_1 \cdots \lambda_j^2 \cdots \lambda_n (1 -
\lambda_1 - \lambda_2 - \cdots - \lambda_n) \in \lambda_n
\mathcal{P}_{n+1}(T),
$$ 
we immediately know that 
\begin{equation} \label{equ:key-observation}
\text{if }c_{i_1, \cdots,i_{n-1},0} \text{ is nonzero, then }
i_k \geq 1 ~(1\leq k \leq n-1).
\end{equation}
Therefore, there are only two
cases in which $c_{i_1, \cdots, i_{n-1},0}$ are nonzero,
\begin{enumerate}
\item $i_k = 3$, the other indexes are $1$. In such case, from
\eqref{equ:m-derivatives2}, we immediately have  
$$ 
3!c_{i_1,\cdots,i_{n-1},0} -
3! \frac{\theta_k}{|T|}\int_T \lambda_n = 0,
$$ 
which implies that 
\begin{equation} \label{equ:nonzero-c1} 
c_{i_1, \cdots, i_{n-1},0} = \frac{\theta_k}{n+1} \qquad i_l=
\begin{cases}
3 & l=k, \\
1 & \mbox{otherwise}.
\end{cases}
\end{equation} 

\item $i_{k_1} = i_{k_2} = 2 ~(k_1 < k_2)$, the other indexes are $1$.
In such case, we have 
$$ 
2!2!c_{i_1,\cdots,i_{n-1},0} - 2!2!\frac{\theta_{k_1} +
  \theta_{k_2}}{|T|}\int_T \lambda_n = 0 
$$
which implies that  
\begin{equation} \label{equ:nonzero-c2}
c_{i_1, \cdots, i_{n-1},0} = \frac{\theta_{k_1} + \theta_{k_2}}{n+1}
\qquad i_l = 
\begin{cases}
2 & l=k_1 \mbox{~or~} k_2 ~(k_1 < k_2),\\
1 & \mbox{otherwise}.
\end{cases}
\end{equation}
\end{enumerate}
To summarize, we have, on $F = F_n$
$$ 
\begin{aligned}
v|_{F} &= \tilde{v}_n|_{F} + \frac{ \lambda_1^F \cdots\lambda_{n-1}^F
}{n+1} \left[\sum_{k=1}^{n-1} \theta_k (\lambda_k^{F})^2 +
\sum_{1\leq k_1 < k_2 \leq n-1}
(\theta_{k_1}+\theta_{k_2})\lambda_{k_1}^F \lambda_{k_2}^F \right] \\ 
  &= \tilde{v}_n|_F + \frac{\lambda_1^F\cdots\lambda_{n-1}^F}{n+1}
  \sum_{k=1}^{n-1} \theta_k \lambda_k^F (\lambda_1^F + \cdots
      +\lambda_{n-1}^F) \\
  &= \tilde{v}_n|_F + \frac{\lambda_1^F\cdots\lambda_{n-1}^F}{n+1}
  \sum_{k=1}^{n-1} \theta_k \lambda_k^F (1-\lambda_{n+1}^F) \qquad
  (\text{Note that }q_F = \lambda_1^F \cdots\lambda_{n-1}^F\lambda_{n+1}^F) \\
  &= \tilde{v}_n|_F + \frac{\lambda_1^F \cdots \lambda_{n-1}^F -
    q_F}{n+1} \sum_{k=1}^{n-1} \theta_k \lambda_k^F \in \mathcal{P}_{n}(F)
      +q_{F}\mathcal{P}_1(F).
\end{aligned}
$$ 
Then, we finish the proof.
\end{proof}

Thanks to the above lemma, we can prove the unisolvent property of the
new nonconforming finite elements by induction on the dimensions. 

\begin{theorem} \label{thm:unisolvent}
For any $n \geq 1$, $D_T^{(n+1,n)}$ is $P_T^{(n+1,n)}$-unisolvent.
\end{theorem}
\begin{proof}
As the dimension of $P_T^{(n+1,n)}$ is the same as the number of
local degrees of freedom, it suffices to show that $v=0$ if all the
degrees of freedom vanish. 

For $n = 1$, the element is an $H^2$ conforming $\mathcal{P}_3$
element in 1D, which means that the unisolvent property holds for
$n=1$. By induction hypothesis and Lemma \ref{lem:similarity-space},
we have $v \in q_T \mathcal{P}_1(T)$ if all the degrees of freedom are
zero. Further, similar to the argument in Lemma
\ref{lem:similarity-space}, $v = \theta_0 q_T + \sum_{j=1}^n \theta_j
\lambda_j q_T$ can be written as 
$$ 
v = \theta_0 \lambda_1 \cdots \lambda_n - \sum_{j=1}^n \theta_0
\lambda_j \lambda_1\cdots \lambda_n + \sum_{j=1}^n \theta_j \lambda_j
q_T.
$$ 
From \eqref{equ:nonzero-c1}, for $1\leq k \leq n$, we obtain 
$$ 
\theta_k = (n+1) c_{i_1,\cdots, i_n, 0} = 0
\qquad i_l =
\begin{cases}
3 & l=k, \\
1 & \mbox{otherwise},
\end{cases}
$$ 
which implies that $v \in \mathrm{span}\{q_T\}$. Therefore, $v = 0$ from
\eqref{equ:m-derivatives2}. 
\end{proof}

We note that the unisolvent property of the new nonconforming finite
elements comes from the {\it similarity} of both shape function
and degrees of freedom.  The similarity of shape function means that
the restriction of function on the sub-simplex belongs to the shape
function space of the corresponding element on the sub-simplex when all the
degrees of freedom vanish, as shown in Lemma
\ref{lem:similarity-space}.  The similarity of degrees of freedom
means that the restriction of degrees of freedom on the sub-simplex
belongs to the degrees of freedom of the corresponding element on the
sub-simplex.  These two similarities, which hold for all $m \leq n+1$
in \cite{wang2013minimal} and in this paper, would lead to the
unisolvent property in general.

\subsection{Canonical nodal interpolation}
Based on Theorem \ref{thm:unisolvent}, we can define the interpolation
operator $\Pi_T: H^{n+1}(T) \mapsto P_T^{(n+1,n)}$ by 
\begin{equation} \label{equ:interpolation}
\Pi_T v := \sum_{i=1}^J p_i d_{T,i}(v) \qquad \forall v\in H^{n+1}(T),
\end{equation}
where $p_i \in P_T^{(n+1,n)}$ is the nodal basis function that
satisfies $d_{T,j}(p_i) = \delta_{ij}$, and $\delta_{ij}$ is the
Kronecker delta.  We emphasize here that the operator
$\Pi_T$ is well-defined for all functions in $H^{n+1}(T)$. 

The following error estimate of the interpolation operator can be
obtained by the standard interpolation theory (cf.
\cite{ciarlet1978finite, brenner2007mathematical}).

\begin{lemma} \label{lem:interpolation}
For $s \in [0,1]$ and $m=n+1$, it holds that, for any
integer $0\leq k \leq m$,  
\begin{equation} \label{equ:interpolation-error}
|v - \Pi_T v|_{k,T} \lesssim h_T^{m+s-k}|v|_{m+s,T} \qquad \forall v\in H^{m+s}(T),
\end{equation}
for all shape-regular $n$-simplex $T$.
\end{lemma}

\subsection{Global finite element spaces} \label{subsec:global}
We define the piecewise polynomial spaces $V_h^{(n+1, n)}$ and
$V_{h0}^{(n+1,n)}$ as follows: 
\begin{itemize}
\item $V_h^{(n+1,n)}$ consists of all functions $v_h|_T \in
P_{T}^{(n+1,n)}$, such that 
\begin{enumerate} 
\item For any $k\in \{1, \cdots, n\}$, any
$(n-k)$-dimensional sub-simplex $F$ of any $T\in \mathcal{T}_h$ and any
$\balpha \in A_k$ with $|\balpha|=n+1-k$, $d_{T,F,\balpha}(v_h)$ is
continuous through $F$.
\item $d_{T,a,0}(v_h)$ is continuous at any vertex $a$. 
\end{enumerate}

\item $V_{h0}^{(n+1,n)} \subset V_h^{(n+1,n)}$ such that for any $v_h
\in V_{h0}^{(n+1,n)}$,
\begin{enumerate}
\item $d_{T,F,\balpha}(v_h) = 0$ if the $(n-k)$-dimensional sub-simplex
$F\subset \partial \Omega$,
\item $d_{T,a,0}(v_h) = 0$ if the vertex $a\in \partial\Omega$.
\end{enumerate}
\end{itemize}

The global interpolation operator $\Pi_h$ on $H^{m}(\Omega)$ is
defined as follows: 
\begin{equation} \label{equ:interpolation-global}
(\Pi_h v)|_T := \Pi_T(v|_T) \qquad \forall T\in \mathcal{T}_h, v \in
H^{m}(\Omega).
\end{equation}
By the above definition, we have $\Pi_hv \in V_h^{(n+1,n)}$ for any $v\in
H^{n+1}(\Omega)$ and $\Pi_h v \in V_{h0}^{(n+1,n)}$ for any $v\in
H_0^{n+1}(\Omega)$. The approximate property of $V_h^{(n+1,n)}$ and
$V_{h0}^{(n+1,n)}$ then follows directly from Lemma
\ref{lem:interpolation}. 

\begin{theorem} \label{thm:interpolation-global}
For $s \in [0,1]$ and $m=n+1$, it holds that 
\begin{equation} \label{equ:interpolation-global-error}
\|v - \Pi_h v\|_{m,h} \lesssim h^{s}|v|_{m+s,\Omega} \qquad \forall
v\in H^{m+s}(\Omega),
\end{equation}
and for any $v \in H^{m}(\Omega)$,
\begin{equation} \label{equ:approximability} 
\lim_{h\to 0}\|v - \Pi_h v\|_{m,h} = 0.
\end{equation}
\end{theorem}
\begin{proof}
The proof of \eqref{equ:approximability} follows the same argument in
\cite[Theorem 2.1]{wang2013minimal} and is therefore omitted here.
\end{proof}

The following lemma can be obtained directly by Lemma
\ref{lem:vanish-derivatives}. 

\begin{lemma} \label{lem:face-average}
Let $1\leq k \leq n$ and $F$ be an $(n-k)$-dimensional sub-simplex of
$T\in \mathcal{T}_h$. Then, for any $v_h \in V_h^{(n+1,n)}$ and
any $T'\in \mathcal{T}_h$ with $F\subset T'$,
\begin{equation} \label{equ:face-average}
\int_F \partial^{\balpha}(v_h|_{T}) = \int_F
\partial^{\balpha}(v_h|_{T'}) \qquad |\balpha| =  
\begin{cases}
n+1-k & k < n, \\
0,1 & k=n.
\end{cases} 
\end{equation}
If $F \subset \partial \Omega$, then for any $v_h \in
V_{h0}^{(n+1,n)}$,
\begin{equation} \label{equ:face-average0}
\int_F \partial^{\balpha}(v_h|_{T}) = 0 \qquad |\balpha| =  
\begin{cases}
n+1-k & k < n, \\
0,1 & k=n.
\end{cases}
\end{equation}
\end{lemma}

\subsection{Weak continuity} 
We note that the conformity of the proposed finite elements is
decreasing with the growth of the dimension.  In fact, $V_h^{(2,1)}$
is the subset of $H^2(\Omega)$ in 1D, and $V_h^{(3,2)}$ is the subset
of $H^1(\Omega)$ in 2D. When $n>2$, a function in $V_h^{(n+1,n)}$
cannot even be continuous, while it holds the weak continuity.  From
Lemma \ref{lem:face-average}, we have the following lemma.

\begin{lemma} \label{lem:weak-continuity}
For $m=n+1$, let $|\balpha|<m$ and $F$ be an $(n-1)$-dimension sub-simplex of
$T\in \mathcal{T}_h$. Then, for any $v_h \in V_h^{(n+1,n)}$,
$\partial_h^{\balpha}v_h$ is continuous at a point on $F$ at least.
If $F\subset \partial \Omega$, then $\partial_h^{\balpha}v_h$ vanishes
at a point on $F$ at least.  
\end{lemma}

The properties in Lemma \ref{lem:weak-continuity} are called {\it weak
continuity} for $V_h^{(n+1,n)}$ and {\it weak zero-boundary
condition} for $V_{h0}^{(n+1,n)}$. Let $S_h^l$ be the
$\mathcal{P}_l$-Lagrange space on $\mathcal{T}_h$ (cf.
\cite{ciarlet1978finite, brenner2007mathematical}), and $\Xi_T^l$ be
the set of nodal points on $T$. Setting 
$$ 
W_h := \{w \in L^2(\Omega)~:~w|_T \in C^\infty(T),~\forall T\in
\mathcal{T}_h\},
$$ 
we define the operator $\Pi_h^{p,l}: W_h \mapsto S_h^l$ as
follows: For all $T\in \mathcal{T}_h$, $\Pi_h^{p,l}|_T \in
\mathcal{P}_l(T)$, and for each $x\in \Xi_T^l$,
\begin{equation} \label{equ:Pl-Lagrangian} 
\Pi_h^{p,l}v(x) :=
\frac{1}{N_h(x)} \sum_{T'\in \mathcal{T}_h(x)} v|_{T'}(x). 
\end{equation}
where $\mathcal{T}_h(x) = \{T'\in \mathcal{T}_h~:~x\in T'\}$ and
$N_h(x) = \# \mathcal{T}_h(x)$. Further, let $S_{h0}^l = S_h^l \cap
H_0^1(\Omega)$, then the operator $\Pi_{h0}^{p,l}: W_h \mapsto
S_{h0}^l$ is defined for each $x\in \Xi_T^l$ as, 
\begin{equation} \label{equ:Pl0-Lagrangian}
\Pi_{h0}^{p,l}v(x) := 
\begin{cases} 
0 & x \in \partial \Omega, \\
\Pi_{h}^{p,l}v(x) & \mbox{otherwise}. 
\end{cases}
\end{equation}
Following the argument in \cite{wang2001necessity}, we have the
following lemma.
\begin{lemma} \label{lem:H1-weak-approximation}
For any $v_h \in V_h^{(n+1,n)}$ and $|\balpha| < m=n+1$, $v_{\balpha} :=
\Pi_h^{p,m+1-|\balpha|} (\partial^\balpha_h v_h) \in H^1(\Omega)$
satisfies 
\begin{equation} \label{equ:H1-weak-approximation} 
|\partial_h^\balpha v_h - v_\balpha|_{j,h} \lesssim h^{m-|\balpha|-j}
|v_h|_{m,h} \qquad 0 \leq j \leq m-|\balpha|.
\end{equation} 
Further, when $v_h\in V_{h0}^{(n+1,n)}$,
\eqref{equ:H1-weak-approximation} holds when $v_\balpha :=
\Pi_{h0}^{p,m+1-|\balpha|} (\partial^\balpha v_h) \in H_0^1(\Omega)$. 
\end{lemma}
\begin{proof}
The proof follows a similar argument in \cite[Lemma
3.1]{wang2013minimal} and is therefore omitted here.
\end{proof}

Thanks to the weak continuity, the Poincar\'{e} inequalities for the
new nonconforming finite elements can be obtained. 
\begin{theorem} \label{thm:Poincare}
The following Poincar\'{e} inequalities hold for $m=n+1$: 
\begin{equation} \label{equ:Poincare}
\begin{aligned}
\|v_h\|_{m,h} &\lesssim |v_h|_{m,h} \qquad \forall v_h \in
V_{h0}^{(n+1,n)}, \\
\|v_h\|_{m,h}^2 &\lesssim |v_h|_{m,h}^2 + \sum_{|\balpha|<m}
\left( \int_\Omega \partial_h^\balpha v_h\right)^2 \qquad \forall v_h
\in V_h^{(n+1,n)}.
\end{aligned}
\end{equation}
\end{theorem}
\begin{proof}
The proof can be found in \cite[Theorem 3.1]{wang2013minimal}.
\end{proof}


\section{Convergence analysis and error estimate} \label{sec:error-estimate}
In this section, we give the convergence analysis of the new
nonconforming finite elements as well as the error estimate under the
broken $H^m$ norm when $m=n+1$. The analysis in some sense is
standard. 

For simplicity, we establish the convergence analysis and error
estimate on the $m$-harmonic equations with homogeneous boundary
conditions: 
\begin{equation} \label{equ:m-harmonic}
\left\{
\begin{aligned}
(-\Delta)^m u &= f \qquad \mbox{in }\Omega, \\
\frac{\partial^k u}{\partial \nu^k} &= 0 \qquad \mbox{on }\partial
\Omega, \quad 0 \leq k \leq m-1.
\end{aligned}
\right.
\end{equation}
The variational problem of \eqref{equ:m-harmonic} can be written as
follows: Find $u \in H_0^m(\Omega)$, such that  
$$ 
a(u,v) = (f,v) \qquad \forall v \in H_0^m(\Omega),
$$ 
where 
\begin{equation} \label{equ:bilinear-form}
a(v,w) := (\nabla^m v, \nabla^m w) = \int_\Omega \sum_{|\balpha|=m}
\partial^\balpha v \partial^\balpha w 
  \qquad \forall v,w\in H^m(\Omega).
\end{equation}

We denote $V_h = V_{h0}^{(m,n)}$ as the nonconforming approximation of
$H_0^{m}(\Omega)$, where $V_{h0}^{(m,n)}$ stands for the new
nonconforming finite elements where $m=n+1$. Then, the nonconforming
finite element method for problem \eqref{equ:m-harmonic} is to find
$u_h \in V_h$, such that 
\begin{equation} \label{equ:nonconforming-FEM}
a_h(u_h, v_h) = (f, v_h) \qquad \forall v_h \in V_h.
\end{equation}
Here, the broken bilinear form $a_h(\cdot, \cdot)$ is defined as 
$$ 
a_h(v, w) := (\nabla_h^m v, \nabla_h^m w) = \sum_{T\in \mathcal{T}_h}
\int_T \sum_{|\balpha|=m} \partial^\balpha v \partial^\balpha w 
\qquad \forall v,w\in H^m(\Omega) + V_h.
$$ 
Given $|\balpha| = m$, it can be written as $\balpha = \sum_{i=1}^m
\be_{j_{\balpha, i}}$, where $\be_i ~(i=1,\cdots, n)$ are the unit
vectors in $\mathbb{R}^n$. We also set $ \balpha_{(k)} = \sum_{i=1}^k
\be_{j_{\balpha,i}}$.

From Theorem \ref{thm:Poincare}, the bilinear form $a_h(\cdot, \cdot)$
is uniformly $V_h$-elliptic. For the consistent condition, we apply the
generalized patch test proposed in \cite{stummel1979generalized} to
obtain the following theorem. Other sufficient conditions that are
easier to achieve can also be used, such as the patch test
\cite{bazeley1965triangular, irons1972experience,
veubeke1974variational, wang2001necessity}, the weak patch test
\cite{wang2001necessity}, and the F-E-M test \cite{shi1987fem,
hu2014new}. 

\begin{theorem} \label{thm:convergence}
For any $f\in L^2(\Omega)$, the solution $u_h$ of problem
\eqref{equ:nonconforming-FEM} converges to the solution of
\eqref{equ:m-harmonic} when $m = n+1$: 
$$ 
\lim_{h \to 0} \|u-u_h\|_{m,h} = 0.
$$ 
\end{theorem}
\begin{proof}
The approximability of $V_h$ is given in Theorem
\ref{thm:interpolation-global}, and the consistent condition can be
verified similar to the Theorem 3.2 in \cite{wang2013minimal} thanks
to the weak continuity and Lemma \ref{lem:face-average}.
\end{proof}

Based on the Strang's Lemma, we have 
\begin{equation} \label{equ:Strang}
|u - u_h|_{m,h} \lesssim \inf_{v_h \in V_h} |u - v_h|_{m,h} + \sup_{v_h\in
V_h} \frac{|a_h(u, v_h) - (f, v_h)|}{|v_h|_{m,h}}.
\end{equation}
The first term on the right-hand side is the approximation error term,
which can be estimated by Theorem \ref{thm:interpolation-global}.
Next, we consider the estimate for the consistent error term.  

\subsection{Error estimate under the extra regularity assumption} In
this subsection, we present the error estimate of the nonconforming
finite element \eqref{equ:nonconforming-FEM} under the extra
regularity assumption, namely $u\in H^{2m-1}(\Omega)$ when
$m=n+1$. We have the following lemma. 

\begin{lemma} \label{lem:nonconforming-regularity}
If $u \in H^{2m-1}(\Omega)$ and $f \in
L^2(\Omega)$, then 
\begin{equation} \label{equ:nonconforming-regularity}
\sup_{v_h \in V_h} \frac{|a_h(u, v_h) - (f, v_h)|}{|v_h|_{m,h}}
\lesssim \sum_{k=1}^{m-1} h^k |u|_{m+k} + h^m \|f\|_0. 
\end{equation}
\end{lemma}

\begin{proof}
The proof follows the same argument in \cite[Lemma
3.2]{wang2013minimal}, so we only sketch the main points. First, we
have 
$$ 
\begin{aligned}
a_h(u, v_h) - (f, v_h) &= \sum_{T\in \mathcal{T}_h} \int_T \left(
\sum_{|\balpha| = m} \partial^\balpha u \partial^\balpha v_h \right) -
(f, v_h) \\
& = \sum_{|\balpha| = m} \sum_{T\in \mathcal{T}_h} \int_T
\partial^\balpha u \partial^\balpha v_h - (-1)^m(\partial^{2\balpha}u) v_h :=
E_1 + E_2 + E_3,
\end{aligned}
$$ 
where 
$$ 
\begin{aligned}
E_1 &:= \sum_{|\balpha| = m} \sum_{T\in \mathcal{T}_h} \int_T 
\partial^\balpha u \partial^\balpha v_h + \partial^{\balpha+\balpha_{(1)}}
u \partial^{\balpha - \balpha_{(1)}} v_h, \\
E_2 &:= \sum_{k=1}^{m-2} (-1)^k \sum_{|\balpha| = m} \sum_{T\in
  \mathcal{T}_h} \int_T \partial^{\balpha + \balpha_{(k)}} u
  \partial^{\balpha - \balpha_{(k)}} v_h + \partial^{\balpha +
    \balpha_{(k+1)}}u \partial^{\balpha - \balpha_{(k+1)}} v_h, \\ 
E_3 &:= (-1)^{m-1}\sum_{|\balpha|=m} \sum_{T \in \mathcal{T}_h} \int_T
\partial^{2\balpha - \be_{j_{\balpha,m}}} u \partial^{\be_{j_{\balpha,m}}}
v_h + (\partial^{2\balpha} u) v_h.
\end{aligned}
$$ 
By Lemma \ref{lem:face-average} and Green's formula, we have   
$$ 
\begin{aligned}
E_1 &= \sum_{|\balpha|=m} \sum_{T \in \mathcal{T}_h} \int_{\partial T}
\partial^\balpha u \partial^{\balpha - \be_{j_{\balpha,1}}} v_h
\nu_{j_{\balpha},1} \\
&= \sum_{|\balpha|=m} \sum_{T\in \mathcal{T}_h} \sum_{F\subset \partial
T} \int_F \left(\partial^\balpha u - P_F^0 \partial^\balpha u \right)
(\partial_h^{\balpha - \be_{j_{\balpha,1}}}v_h - P_F^0 \partial_h^{\balpha -
 \be_{j_{\balpha,1}}}v_h)\nu_{j_{\balpha},1},
\end{aligned}
$$ 
where $P_F^0: L^2(F) \mapsto \mathcal{P}_0(F)$ is the orthogonal
projection, $\nu = (\nu_1, \cdots, \nu_n)$ is the unit outer normal to
$\partial T$. Using the Schwarz inequality and the interpolation theory, we
obtain 
\begin{equation} \label{equ:regularity-E1}
|E_1| \lesssim h|u|_{m+1} |v_h|_{m,h}.
\end{equation} 

When $m>1$, let $v_\bbeta \in H_0^1(\Omega)$ be the piecewise
polynomial as in Lemma \ref{lem:H1-weak-approximation}. Then, 
Green's formula leads to 
$$ 
\begin{aligned}
E_2 &= \sum_{k=1}^{m-2} (-1)^k \sum_{|\balpha| = m} \sum_{T\in
  \mathcal{T}_h} \int_T \partial^{\balpha + \balpha_{(k)}} u
  \partial^{\be_{j_{\balpha},k+1}} (\partial^{\balpha-\balpha_{(k+1)}}v_h
   - v_{\balpha-\balpha_{(k+1)}}) \\
    &+ \sum_{k=1}^{m-2} (-1)^k \sum_{|\balpha| = m} \sum_{T\in
  \mathcal{T}_h} \int_T \partial^{\balpha + \balpha_{(k+1)}}u  (
      \partial^{\balpha - \balpha_{(k+1)}} v_h -
      v_{\balpha-\balpha_{(k+1)}} ), \\ 
\end{aligned}
$$ 
which implies 
\begin{equation} \label{equ:regularity-E2}
|E_2| \lesssim \sum_{k=1}^{m-2} h^k|u|_{m+k}|v_h|_{m,h} +
h^{k+1}|u|_{m+k+1} |v_h|_{m,h}.
\end{equation}
Finally, we have 
$$
\begin{aligned}
E_3 &= (-1)^{m-1}\sum_{|\balpha|=m} \sum_{T \in \mathcal{T}_h} \int_T
\partial^{2\balpha - \be_{j_{\balpha,m}}} u \partial^{\be_{j_{\balpha,m}}}
(v_h-v_0) + (\partial^{2\balpha} u) (v_h-v_0),
\end{aligned}
$$
which gives 
\begin{equation} \label{equ:regularity-E3} 
|E_3| \lesssim h^{m-1}|u|_{2m-1}|v_h|_{m,h} + h^m \|f\|_0 |v_h|_{m,h}.
\end{equation} 
By the estimates \eqref{equ:regularity-E1}, \eqref{equ:regularity-E2},
and \eqref{equ:regularity-E3}, we obtain the desired estimate
\eqref{equ:nonconforming-regularity}.
\end{proof}

From Lemma \ref{lem:nonconforming-regularity}, we have the following
theorem.
\begin{theorem} \label{thm:error-regularity}
If $u \in H^{2m-1}(\Omega) \cap H_0^m(\Omega)$
and $f \in L^2(\Omega)$, then 
\begin{equation} \label{equ:error-regularity}
|u - u_h|_{m,h} \lesssim \sum_{k=1}^{m-1} h^k |u|_{m+k} +
h^m \|f\|_0. 
\end{equation}
\end{theorem}

\begin{remark}
From the proof of Lemma \ref{lem:nonconforming-regularity}, the error
estimate can be improved in the following cases: 
\begin{enumerate}
\item $n=1, m=2$: If $u \in H^3(\Omega)$, then  
$$ 
|u-u_h|_{2,h} \lesssim h|u|_{3}.
$$ 
\item $n=2, m=3$: We have $V_h \subset H_0^1(\Omega)$, and if $u\in
H^5(\Omega)$, then  
$$ 
|u-u_h|_{3,h} \lesssim h|u|_{4} + h^2|u|_{5}. 
$$ 
\end{enumerate}
\end{remark}

\begin{remark}
In 2D, since the $H^3$ nonconforming finite element space satisfies
$V_h^{(3,2)} \subset H^1(\Omega)$, then $V_h^{(3,2)}$ is robust for
the singularly perturbed problem $-\varepsilon^2\Delta^3 u - \Delta u =
f$. The proof follows a similar technique developed in
\cite{nilssen2001robust} and Lemma \ref{lem:nonconforming-regularity}
and is therefore omitted here. Further, a modified $H^3$ nonconforming
element that converges for both second and fourth order elliptic
problems is given in Section \ref{sec:perturbed-2D}.
\end{remark}

\subsection{Error estimate by conforming relatives}
The error estimate can be improved with minimal regularity under the
following assumption, which is motivated by the
conforming relatives proposed by Brenner (cf. \cite{brenner1996two}). 
\begin{assumption}[Conforming relatives] \label{asm:conforming-relatives}
There exists an $H^m$ conforming finite element space $V_h^c \subset
H_0^m(\Omega)$, and an operator $\Pi_h^c: V_h \mapsto V_h^c$ such that 
\begin{equation} \label{equ:conforming-relatives}
\sum_{j=0}^{m-1} h^{2(j-m)} |v_h - \Pi_h^c v_h|_{j,h}^2 + |\Pi_h^c
v_h|_{m,h}^2 \lesssim |v_h|_{m,h}^2. 
\end{equation} 
\end{assumption}
The above assumption has been verified for the various cases; see
\cite{scott1990finite, brenner2003poincare} for the case in which
$m=1$, \cite{brenner1996two, li2014new} for the Morley element in 2D,
and \cite{hu2014new, hu2016canonical} for arbitrary $m \geq 1$ in 2D. 

Let $\mathcal{P}_0(\mathcal{T}_h)$ be the piecewise constant space on
$\mathcal{T}_h$.  To obtain the quasi-optimal error estimate under 
Assumption \ref{asm:conforming-relatives}, we first define the
piecewise constant projection $P_h^0: L^2(\Omega) \mapsto
\mathcal{P}_0(\mathcal{T}_h)$ as 
\begin{equation} \label{equ:constant-proj}
P_h^0 v|_T := \frac{1}{|T|} \int_T v \qquad \forall T \in \mathcal{T}_h.
\end{equation}
For any $F\in \mathcal{F}_h$, let $\omega_F$ be the union of all
elements that share the face $F$. We further define the average operator
on $\omega_F$ as  
\begin{equation} \label{equ:average-omega}
P_{\omega_F}^0v := \frac{1}{|\omega_F|} \int_{\omega_F} v.
\end{equation}
Following the standard DG notation (cf. \cite{arnold2002unified}),
$\llbracket\cdot \rrbracket$ and $\{\cdot\}$ represent the jump and
average operators, respectively.

\begin{lemma} \label{lem:nonconforming-no-regularity}
Under Assumption \ref{asm:conforming-relatives}, if $f \in
L^2(\Omega)$, then 
\begin{equation} \label{equ:nonconforming-no-regularity}
\begin{aligned}
\sup_{v_h \in V_h} \frac{|a_h(u, v_h) - \langle f, v_h \rangle
|}{|v_h|_{m,h}} &\lesssim \inf_{w_h\in V_h} |u-w_h|_{m,h} +
 h^m \|f\|_0 \\
&+ \sum_{|\balpha|=m} \left(\|\partial^\balpha u - P_h^0
    \partial^\balpha u\|_0 + \sum_{F\in \mathcal{F}_h}
    \|\partial^\balpha u - P_{\omega_F}^0 \partial^\balpha
    u\|_{0,\omega_F} \right). 
\end{aligned}
\end{equation}
\end{lemma}

\begin{proof}
For any $w_h \in V_h$,
$$ 
\begin{aligned}
a_h(u, v_h) - (f,v_h) &= a_h(u, v_h - \Pi_h^c v_h) - (f, v_h -
\Pi_h^c v_h) \\
&= a_h(u - w_h, v_h - \Pi_h^c v_h) + a_h(w_h, v_h - \Pi_h^c v_h) - (f,
v_h - \Pi_h^c v_h)
\end{aligned}
$$ 
For the first and third terms, we have 
\begin{equation} \label{equ:no-regularity1} 
\begin{aligned}
|a_h(u-w_h, v_h - \Pi_h^c v_h)| &\lesssim |u-w_h|_{m,h}|v_h - \Pi_h^c
v_h|_{m,h} \lesssim |u-w_h|_{m,h} |v_h|_{m,h}, \\
|(f, v_h - \Pi_h^c v_h)| & \lesssim \|f\|_{0} \|v_h - \Pi_h^c
v_h\|_0 \lesssim h^m\|f\|_{0} |v_h|_{m,h}.
\end{aligned}
\end{equation}
Next, we estimate the second term. First,
$$ 
a_h(w_h, v_h - \Pi_h^c v_h) = \sum_{|\balpha|=m}\sum_{T\in
  \mathcal{T}_h} \int_T \partial^\balpha w_h \partial^\balpha (v_h -
      \Pi_h^c v_h) := E_1 + E_2, 
$$ 
where
$$ 
\begin{aligned}
E_1 &:= \sum_{|\balpha| = m} \sum_{T\in \mathcal{T}_h} \int_T 
\partial^\balpha w_h \partial^\balpha (v_h-\Pi_h^cv_h) +
\partial^{\balpha+\be_{j_{\balpha},1}} w_h \partial^{\balpha -
  \be_{j_{\balpha},1}}
(v_h - \Pi_h^c v_h), \\
E_2 &:= -\sum_{|\balpha| = m} \sum_{T\in
  \mathcal{T}_h} \int_T \partial^{\balpha + \be_{j_\balpha,1}}w_h
  \partial^{\balpha - \be_{j_\balpha,1}} (v_h-\Pi_h^c v_h). 
\end{aligned}
$$ 
By the Lemma \ref{lem:face-average} and Green's formula, we have   
$$ 
\begin{aligned}
E_1 &= \sum_{|\balpha|=m} \sum_{T \in \mathcal{T}_h} \int_{\partial T}
\partial_h^\balpha w_h \partial_h^{\balpha - \be_{j_\balpha,1}} (v_h -
\Pi_h^c v_h) \nu_{j_{\balpha,1}}\\
&= \sum_{|\balpha|=m} \sum_{F\in \mathcal{F}_h} \int_F
\{\partial_h^\balpha w_h\}\llbracket \partial_h^{\balpha - \be_{j_\balpha,1}} (v_h -
    \Pi_h^c v_h) \rrbracket_{j_{\balpha,1}} \\
&+ \sum_{|\balpha|=m} \sum_{F\in \mathcal{F}_h^i} \int_F
\llbracket\partial_h^\balpha w_h\rrbracket_{j_{\balpha,1}} 
\{\partial_h^{\balpha-\be_{j_\balpha,1}} (v_h - \Pi_h^c v_h)\} \qquad
(\text{Eq. (3.3) in \cite{arnold2002unified}})\\
&= \sum_{|\balpha|=m} \sum_{F\in \mathcal{F}_h} \int_F
\{\partial_h^\balpha w_h - P_h^0 \partial^\balpha u\}\llbracket
\partial_h^{\balpha - \be_{j_\balpha,1}} (v_h - \Pi_h^c
    v_h)\rrbracket_{j_{\balpha,1}} \\
&+ \sum_{|\balpha|=m} \sum_{F\in \mathcal{F}_h^i} \int_F \llbracket \partial_h^\balpha
w_h - P_{\omega_F}^0 \partial^\balpha u\rrbracket_{j_{\balpha,1}}
\{\partial_h^{\balpha-\be_{j_\balpha,1}} (v_h - \Pi_h^c v_h)\}.
\end{aligned}
$$ 
Therefore, it follows from the trace inequality and inverse inequality
that 
\begin{equation}\label{equ:no-regularity-E1}
\begin{aligned}
|E_1| &\lesssim \sum_{|\balpha|=m} \sum_{F\in \mathcal{F}_h} h_F^{-1}
\|\partial_h^\balpha w_h - P_h^0 \partial^\balpha u\|_{0,\omega_F} |v_h
- \Pi_h^c v_h|_{m-1,h} \\
&+ \sum_{|\balpha|=m} \sum_{F\in \mathcal{F}_h}h_F^{-1} 
\|\partial_h^\balpha w_h - P_{\omega_F}^0 \partial^\balpha
u\|_{0,\omega_F} |v_h - \Pi_h^c v_h|_{m-1,h} \\
&\lesssim \sum_{|\balpha|=m} \left(\|\partial_h^\balpha w_h - P_h^0
\partial^\balpha u\|_0 + \sum_{F\in \mathcal{F}_h} \|\partial_h^\balpha w_h
- P_{\omega_F}^0 \partial^\balpha u\|_{0,\omega_F} 
\right)|v_h|_{m,h}.
\end{aligned}
\end{equation}

For the estimate of $E_2$, we obtain
$$ 
\begin{aligned}
E_2 &= -\sum_{|\balpha| = m} \sum_{T\in \mathcal{T}_h} \int_T
\partial^{\be_{j_\alpha,1}}(\partial^{\balpha}w_h - P_h^0
  \partial^{\balpha} u) \partial^{\balpha - \be_{j_\alpha,1}}(v_h -
    \Pi_h^c v_h), 
\end{aligned}
$$
which gives 
\begin{equation} \label{equ:no-regularity-E2} 
\begin{aligned}
|E_2| &\lesssim \sum_{|\balpha|=m} h^{-1}
\|\partial_h^\balpha w_h - P_h^0 \partial^\balpha u\|_0 |v_h - \Pi_h^c
v_h|_{m-1,h} \\
&\lesssim \sum_{|\balpha|=m} \|\partial_h^\balpha w_h -
P_h^0\partial^\balpha u\|_0 |v_h|_{m,h}. 
\end{aligned}
\end{equation} 
We therefore complete the proof by \eqref{equ:no-regularity1},
\eqref{equ:no-regularity-E1}, \eqref{equ:no-regularity-E2}, and the
triangle inequality.
\end{proof}

From Lemma \ref{lem:nonconforming-no-regularity}, we have the following
theorem.
\begin{theorem} \label{thm:error-no-regularity}
Under Assumption \ref{asm:conforming-relatives}, if $f \in
L^2(\Omega)$ and $u \in H^{m+t}(\Omega)$, then 
\begin{equation}\label{equ:error-no-regularity}
|u - u_h|_{m,h} \lesssim h^{s}|u|_{m+s} + h^m \|f\|_0,
\end{equation}
where $s = \min\{1,t\}$.
\end{theorem}

\begin{remark}
We note that only $H^m$ regularity is required in Lemma
\ref{lem:nonconforming-no-regularity} and Theorem
\ref{thm:error-no-regularity}. Similar technique can be found in
\cite{mao2010error,hu2014new,hu2016canonical}.
\end{remark}

\section{A robust $H^3$ nonconforming element in 2D}
\label{sec:perturbed-2D}
Taking the cue from the degrees of freedom of the Morley element, we
can obtain an nonconforming finite element space $\tilde{V}_h^{(3,2)}$
that converges for both second and fourth order elliptic problems. The
shape function space on a triangle $T$ is given by 
\begin{equation} \label{equ:perturbed-space}
\tilde{P}_T^{(3,2)} := \mathcal{P}_3(T) + q_T \mathcal{P}_1(T) + q_T^2
\mathcal{P}_1(T).
\end{equation} 
The degrees of freedom for $\tilde{P}_T^{(3,2)}$ are determined by and
depicted as
\begin{multicols}{2}
\begin{subequations}\label{equ:perturbed-DOFs}
\begin{align} 
\frac{1}{|F|}\int_F \frac{\partial^2 v}{\partial \nu_F^2}  
& \quad \text{for all edges }F \label{equ:perturbed-DOF1} \\ 
\nabla v(a) & \quad \text{for all vertices }a
\label{equ:perturbed-DOF2} \\
\frac{1}{|F|}\int_F \frac{\partial v}{\partial \nu_F} & \quad
\text{for all edges }F \label{equ:perturbed-DOF3} \\
v(a) & \quad \text{for all vertices }a
\label{equ:perturbed-DOF4}
\end{align}
\end{subequations}
\qquad \includegraphics[width=1.3in]{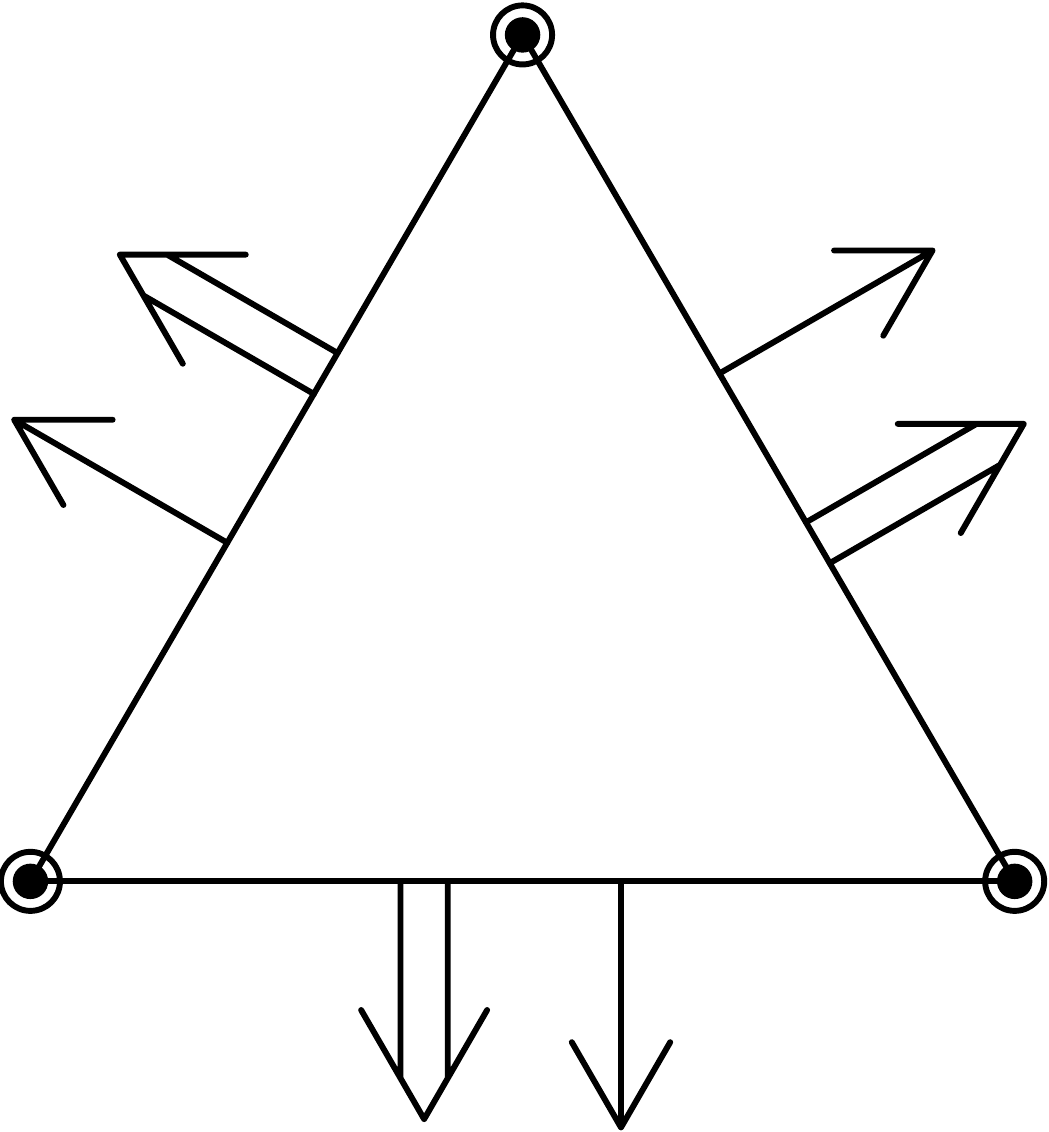} 
\end{multicols}
\begin{lemma} \label{lem:perturbed-unisolvent}
Any function $v \in \tilde{P}_T^{(3,2)}$ is uniquely determined by the
degrees of freedom \eqref{equ:perturbed-DOFs}.
\end{lemma}
\begin{proof}
Clearly, $\dim(\tilde{P}_T^{(3,2)}) = 15$, which is exactly the number
of degrees of freedom given in \eqref{equ:perturbed-DOFs}. If $v \in
\tilde{P}_T^{(3,2)}$ has all the degrees of freedom zero, since $v|_F
\in \mathcal{P}_3(F)$, we obtain that $v$ is of the form $v = q_T p$,
where $p \in \mathcal{P}_1(T) \oplus q_T\mathcal{P}_1(T)$. Let $p =
p_1 + q_T p_2 \in \mathcal{P}_1(T) \oplus q_T\mathcal{P}_1(T)$.
Applying the degrees of freedom \eqref{equ:perturbed-DOF3} to $v =
q_Tp_1 + q_T^2 p_2$, we obtain $p_1 = 0$ by direct calculation. Then,
the vanishing of the degrees of freedom \eqref{equ:perturbed-DOF1}
implies that $p_2 = 0$.   
\end{proof}
The degrees of freedom \eqref{equ:perturbed-DOF3} and
\eqref{equ:perturbed-DOF4} are exactly those of the Morley element.
Similar to Lemma
\ref{lem:face-average}, the global finite element spaces
$\tilde{V}_h^{(3,2)}$ and $\tilde{V}_{h0}^{(3,2)}$, defined in the
same way as $V_h^{(3,2)}$ and $V_{h0}^{(3,2)}$ in Section
\ref{subsec:global}, satisfy the following lemma. 
\begin{lemma} \label{lem:perturbed-face}
Let $F$ be an edge of $T\in \mathcal{T}_h$. For any $v_h \in
\tilde{V}_h^{(3,2)}$ and any $T'\in \mathcal{T}_h$ with $F \subset
T'$, 
\begin{equation}
\int_F \partial^{\balpha}(v_h|_T) = \int_F
  \partial^{\balpha}(v_h|_{T'}) \qquad  
|\balpha| = 1, 2.
\end{equation}
If $F \subset \partial\Omega$, then for any $v_h \in
\tilde{V}_{h0}^{(3,2)}$, 
\begin{equation}
\int_F \partial^{\balpha}(v_h|_T) = 0 \qquad |\balpha| = 1, 2.
\end{equation}
\end{lemma}
Then, a routine argument shows that $\tilde{V}_h^{(3,2)}$ converges
for second, fourth and sixth order elliptic problems. Hence, followed
by a similar argument in \cite{nilssen2001robust}, the modified
$H^3$ nonconforming finite element space is robust for the sixth order
singularly perturbed problems.

An additional advantage of the modified nonconforming element is its
convenience in handling the sixth order equation with the mixed boundary
conditions: 
\begin{equation} \label{equ:m-natural}
\left\{
\begin{aligned}
(-\Delta)^3 u + b_0 u &= f \qquad \mbox{in }\Omega \subset
\mathbb{R}^2, \\
\frac{\partial (\Delta^{k} u)}{\partial \nu} &= 0 \qquad \mbox{on }\partial
\Omega, \quad 0 \leq k \leq 2,
\end{aligned}
\right.
\end{equation}
where $b_0 = \mathcal{O}(1)$ is a positive constant. For any
$F \in \mathcal{F}_h^\partial$, let $\tau$ denote the unit tangential
direction obtain by rotating $\nu$ $90^\circ$ counterclockwise. It is
straightforward to show that 
$$
\nabla v = \frac{\partial v}{\partial \nu} \nu + \frac{\partial
v}{\partial \tau}\tau, \qquad \Delta v = \frac{\partial^2
v}{\partial \nu^2} + \frac{\partial^2 v}{\partial \tau^2}.
$$
Therefore, the variational problem of \eqref{equ:m-natural} reads:
Find $u \in \tilde{H}_0^3(\Omega) :=\{ v\in  H^3(\Omega)~|~
\frac{\partial v}{\partial \nu} = 0~~\text{on }\partial\Omega\}$, such
that
\begin{equation} \label{equ:m-natural-bilinear} 
a(v, w) + b(v, w) = (f, v) \qquad \forall v\in \tilde{H}_0^3(\Omega),
\end{equation} 
where $a(\cdot, \cdot)$ is defined in \eqref{equ:bilinear-form} and
$$ 
b(v,w) := (b_0 v, w) \qquad \forall v, w \in L^2(\Omega).
$$ 
The well-posedness of
\eqref{equ:m-natural-bilinear} follows from the Poincar\'{e}
inequality (obtained by compact embedding argument) 
\begin{equation} \label{equ:H3-poincare}
\|v\|_3^2 \lesssim \|v\|_0^2 + |v|_3^2 \qquad \forall v \in H^3(\Omega).
\end{equation} 

We denote the nonconforming finite element space as 
\begin{equation} \label{equ:m-natural-space}
\begin{aligned}
\tilde{V}_h := \{v_h \in \tilde{V}_h^{(3,2)} ~|~  \frac{1}{|F|}\int_F
\frac{\partial v_h}{\partial \nu} &= 0 \quad \text{for all edges }F
\subset \partial\Omega \\ 
\frac{\partial v_h}{\partial \nu}(a) &= 0 \quad \text{for all
  vertices }a \in \partial \Omega\}.
\end{aligned}
\end{equation}
Then, the nonconforming finite element method for
\eqref{equ:m-natural} is to find $u_h \in \tilde{V}_h$, such that 
\begin{equation} \label{equ:m-natural-scheme} 
a_h(u_h, v_h) + b(u_h, v_h) = (f, v_h) \qquad \forall v_h \in
\tilde{V}_h.
\end{equation} 
Here, $a_h(\cdot,\cdot)$ is the broken bilinear form defined in
\eqref{equ:nonconforming-FEM}. We first establish its well-posedness.
By using the standard enriching operator $E_h$ (cf.
\cite{gudi2011interior}) from $\tilde{V}_h^{(3,2)}$ to the $H^3$
conforming finite element space, e.g. {\v{Z}}en{\'\i}{\v{s}}ek finite
element space \cite{vzenivsek1970interpolation}, we have
\begin{equation} \label{equ:enriching}
\begin{aligned}
& h^{-6}\|v_h - E_h v_h\|_0^2 + h^{-4}|v_h - E_h v_h|_{1,h}^2 \\
+~ & h^{-2}|v_h - E_h v_h|_{2,h}^2 + |v_h - E_h v_h|_{3,h}^2 \lesssim
  |v_h|_{3,h}^2 \qquad \forall v_h \in \tilde{V}_h^{(3,2)}.
\end{aligned}
\end{equation}
The well-posedness of \eqref{equ:m-natural-scheme} follows from the
lemma below.
\begin{lemma} \label{lem:H3h-poincare} 
It holds that 
\begin{equation} \label{equ:H3h-poincare} 
\|v_h\|_{3,h}^2 \lesssim \|v_h\|_0^2 + |v_h|_{3,h}^2 \qquad \forall
v_h \in \tilde{V}_h^{(3,2)}.
\end{equation} 
\end{lemma}
\begin{proof}
It follows from \eqref{equ:H3-poincare} and \eqref{equ:enriching} that 
$$ 
\begin{aligned}
\|v_h\|_{3,h}^2 & \lesssim \|E_h v_h\|_{3}^2 + \|v_h - E_h
v_h\|_{3,h}^2 \lesssim \|E_h v_h\|_{0}^2 + |E_h v_h|_{3}^2 +
|v_h|_{3,h}^2 \\
& \lesssim \|v_h - E_h v_h\|_0^2 + \|v_h\|_0^2 + |v_h
- E_h v_h|_{3,h}^2 + |v_h|_{3,h}^2 \\
& \lesssim \|v_h\|_0^2 + |v_h|_{3,h}^2.
\end{aligned}
$$ 
Then, we finish the proof.
\end{proof}

Next, the consistency error $E(u, v_h)$ can be written as  
$$ 
\begin{aligned}
E(u, v_h) &= a_h(u, v_h) + b(u, v_h) - (f, v_h) \\ 
& = \sum_{T\in \mathcal{T}_h} \int_T \nabla^3u : \nabla^3 v_h +
\int_\Omega (\Delta^3 u) v_h \\
& = \sum_{T\in \mathcal{T}_h} \int_T \nabla^3u : \nabla^3 v_h +
\sum_{T\in \mathcal{T}_h} \int_T \nabla^2(\Delta u) : \nabla^2 v_h \\
& ~ - \sum_{T\in \mathcal{T}_h} \int_T \nabla^2(\Delta u) : \nabla^2
v_h - \sum_{T\in \mathcal{T}_h} \int_T \nabla(\Delta^2 u) \cdot \nabla
v_h \\
& ~ + \sum_{T\in \mathcal{T}_h} \int_T \nabla(\Delta^2 u) \cdot \nabla
v_h + \int_\Omega (\Delta^3 u) v_h \\
& := E_1 + E_2 + E_3.
\end{aligned}
$$ 
Recall that $P_F^0: L^2(F) \mapsto \mathcal{P}_0(F)$ is the orthogonal
projection.  By Green's formula and Lemma \ref{lem:perturbed-face}, we
have 
$$ 
\begin{aligned}
E_1 &= \sum_{T \in \mathcal{T}_h} \int_{\partial T}
\frac{\partial}{\partial \nu}(\nabla^2 u) : \nabla^2 v_h \\
&= \sum_{F\in \mathcal{F}_h^i} \int_F \left( \frac{\partial}{\partial
  \nu}(\nabla^2 u) - P_F^0 \frac{\partial}{\partial
  \nu}(\nabla^2 u)\right): \left(\nabla^2 v_h - P_F^0  \nabla^2 v_h
  \right) \\ 
&~ + \sum_{F\in \mathcal{F}_h^\partial} \int_F \frac{\partial^3
u}{\partial \nu^3} \frac{\partial^2 v_h}{\partial \nu^2} + 2
\frac{\partial^3 u}{\partial \nu^2 \partial \tau} \frac{\partial^2
  v_h}{\partial \nu\partial \tau} + 
\frac{\partial^3
u}{\partial \nu \partial \tau^2} \frac{\partial^2 v_h}{\partial
  \tau^2} \\
&= \sum_{F\in \mathcal{F}_h^i} \int_F \left( \frac{\partial}{\partial
  \nu}(\nabla^2 u) - P_F^0 \frac{\partial}{\partial
  \nu}(\nabla^2 u)\right): \left(\nabla^2 v_h - P_F^0  \nabla^2 v_h
  \right) \\ 
& + 2 \sum_{F\in \mathcal{F}_h^\partial} \int_F \frac{\partial^3
  u}{\partial \nu^2 \partial \tau} \frac{\partial^2 v_h}{\partial
  \nu\partial \tau} \\
&= \sum_{F\in \mathcal{F}_h^i} \int_F \left( \frac{\partial}{\partial
  \nu}(\nabla^2 u) - P_F^0 \frac{\partial}{\partial
  \nu}(\nabla^2 u)\right): \left(\nabla^2 v_h - P_F^0  \nabla^2 v_h
  \right) \\ 
& + 2 \sum_{F\in \mathcal{F}_h^\partial} \int_F \left( \frac{\partial^3
  u}{\partial \nu^2 \partial \tau} - P_F^0 \frac{\partial^3
  u}{\partial \nu^2 \partial \tau} \right) 
  \left( \frac{\partial^2 v_h}{\partial \nu\partial \tau} - 
  P_F^0 \frac{\partial^2 v_h}{\partial \nu\partial \tau} \right).
\end{aligned}
$$ 
Here, we use the fact that for any $F\in \mathcal{F}_h^\partial$,
$\int_F \frac{\partial^2 v_h}{\partial \nu\partial\tau} = 0$.  Then,
the Schwarz inequality and the interpolation theory imply 
$$ 
|E_1| \lesssim h|u|_4 |v_h|_{3,h}.
$$ 
A similar argument shows that 
$$
\begin{aligned}
E_2 &= -\sum_{F\in \mathcal{F}_h^i} \int_F \left( \frac{\partial }{\partial
    \nu} (\Delta \nabla u) - P_F^0 \frac{\partial }{\partial \nu}
    (\Delta \nabla u) \right) \cdot \left( \nabla v_h - P_F^0 \nabla
      v_h \right) \\
    &~~ - \sum_{F\in \mathcal{F}_h^\partial } \int_F
\left(\frac{\partial^2}{\partial \nu^2}(\Delta u) - P_F^0
    \frac{\partial^2}{\partial \nu^2}(\Delta u) \right) \left(
    \frac{\partial v_h}{\partial \nu} - P_F^0  \frac{\partial
    v_h}{\partial \nu} \right), 
\end{aligned}
$$ 
which gives 
$$ 
|E_2| \lesssim h|u|_5|v_h|_{2,h}.
$$ 
The boundary conditions and the $H^1$-conformity of $\tilde{V}_h$
imply that $E_3 = 0$.

By Strang' lemma and interpolation theory, we obtain the following theorem. 
\begin{theorem} \label{thm:perturbed-converge}
If $u \in H^5(\Omega) \cap \tilde{H}_0^3(\Omega)$ and $f\in
  L^2(\Omega)$, then 
$$ 
\|u - u_h\|_{3,h} \lesssim h(|u|_4 + |u|_5).
$$ 
\end{theorem}

\section{Numerical Tests} \label{sec:tests}
In this section, we present several 2D numerical results to support the
theoretical results obtained in Section \ref{sec:error-estimate}.

\subsection{Example 1}
In the first example, we choose $f=0$ so that the exact solution is $u
= \exp(\pi y) \sin(\pi x)$ in $\Omega = (0,1)^2$, which provides the
nonhomogeneous boundary conditions. After computing
\eqref{equ:nonconforming-FEM} for various values of $h$, we calculate
the errors and orders of convergence in $H^{k} (k=0,1,2,3)$ and report
them in Table \ref{tab:example1}. The table shows that the computed
solution converges linearly to the exact solution in the $H^3$ norm, which
is in agreement with Theorem \ref{thm:error-regularity} and
Theorem \ref{thm:error-no-regularity}.  Further, Table
\ref{tab:example1} indicates that $\|u-u_h\|_0$, $|u-u_h|_{1,h}$ and
$|u-u_h|_{2,h}$ are all of order $h^2$.

\begin{table}[!htbp]
\caption{Example 1: Errors and observed convergence orders.}
\centering
{\small{
\begin{tabular}{@{}c|cc|cc|cc|cc@{}}
   \hline
  $1/h$	&$\|u-u_h\|_0$	& Order	& 
  $|u-u_h|_{1,h}$ & Order & 
  $|u-u_h|_{2,h}$	& Order & $|u-u_h|_{3,h}$ & Order\\ \hline
  8		&2.7221e-3 &--	  &3.7562e-2 &--	  &8.1131e-1 &--   &5.0076e+1	&--  \\
  16	&6.5721e-4 &2.05	&6.6469e-3 &2.50	&2.1044e-1 &1.95 &2.5856e+1	&0.95\\
  32	&1.6337e-4 &2.01	&1.4450e-3 &2.20	&5.3510e-2 &1.98 &1.3081e+1	&0.98\\
  64	&4.1029e-5 &1.99	&3.4724e-4 &2.06	&1.3474e-2 &1.99 &6.5673e+0	&0.99\\
  \hline
\end{tabular}}}
\label{tab:example1}
\end{table}

\begin{figure}[!htbp]
\caption{Uniform grids for Example 1 and Example 2.}
\label{fig:uniform-grids}
\centering 
\captionsetup{justification=centering}
\subfloat[Example 1: Unit square domain]{\centering 
   \includegraphics[width=0.4\textwidth]{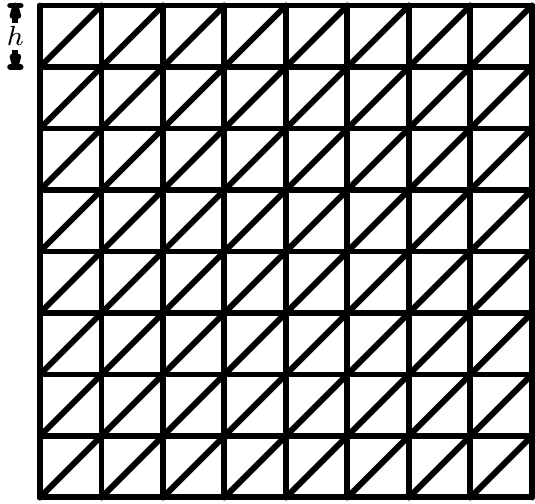} 
   \label{fig:square}
}%
\qquad\qquad  
\subfloat[Example 2: L-shaped domain]{\centering 
   \includegraphics[width=0.4\textwidth]{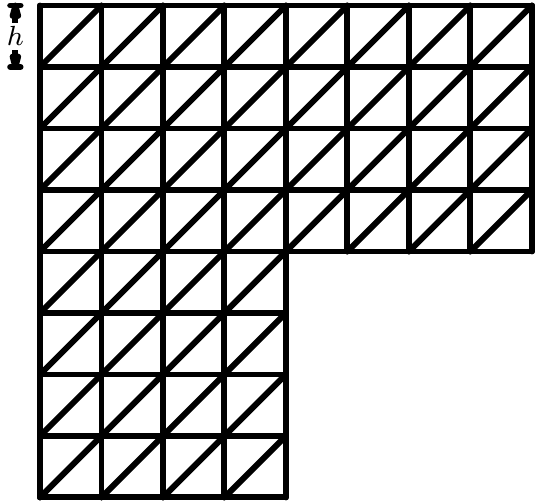} 
   \label{fig:L-shaped}
}
\end{figure}

\subsection{Example 2}
In the second example, we test the method in which the solution has
partial regularity on a non-convex domain. To this end, we solve the
tri-harmonic equation 
$$ 
(-\Delta)^3 u = 0
$$ 
on the 2D L-shaped domain $\Omega = (-1,1)^2\setminus [0,1)\times
(-1,0]$ shown in Figure \ref{fig:L-shaped}, with Dirichlet boundary conditions
given by the exact solution 
$$ 
u = r^{2.5} \sin(2.5\theta),
$$ 
where $(r,\theta)$ are polar coordinates. Due to the singularity at
the origin, the solution $u \in H^{3+1/2}(\Omega)$. The method does
converge with the optimal order $h^{1/2}$ in the broken $H^3$
norm, as shown in Table \ref{tab:example2}.

\begin{table}[!htbp]
\caption{Example 2: Errors and observed convergence orders.}
\centering
{\small{
\begin{tabular}{@{}c|cc|cc|cc|cc@{}}
   \hline
  $1/h$	&$\|u-u_h\|_0$	& Order	& 
  $|u-u_h|_{1,h}$ & Order & 
  $|u-u_h|_{2,h}$	& Order & $|u-u_h|_{3,h}$ & Order\\ \hline
  4		&9.0977e-4 &--	  &6.5652e-3 &--	  &4.9732e-2 &--   &9.3881e-1	&--  \\
  8		&3.3208e-4 &1.45	&2.0825e-3 &1.66	&2.0598e-2 &1.27 &6.8270e-1	&0.46\\
  16	&1.3845e-4 &1.26	&7.6830e-4 &1.44	&8.3676e-3 &1.30 &4.8821e-1	&0.48\\
  32	&6.2963e-5 &1.13	&3.2391e-4 &1.26	&3.4430e-3 &1.28 &3.4697e-1	&0.49\\
  64	&2.9775e-5 &1.08	&1.4691e-4 &1.14	&1.4548e-3 &1.24 &2.4593e-1	&0.50\\
  \hline
\end{tabular}}}
\label{tab:example2}
\end{table}

\section{Concluding Remarks} \label{sec:concluding}
In this paper, we propose and study the nonconforming finite
elements for $2m$-th order elliptic problems when $m=n+1$.
After showing the convergence analysis under minimal regularity
assumption, we provide two kinds of error estimates --- one requires
the extra regularity, and the other assumes only minimal regularity but
the existence of the conforming relative. We also propose an $H^3$
nonconforming finite element space that is robust for the sixth order
singularly perturbed problems in 2D.

The universal construction when $m=n+1$ is motivated by the similarity
properties of both shape function spaces and degrees of freedom, which
also work for the WMX elements that require $m \leq n$. However, the
universal construction when $m \geq n+2$ is still an open problem.

\subsection*{Acknowledgement} The authors would like to express their
gratitude to anonymous referees for the valuable comments leading to a
better version of this paper.

\appendix 

%

\bibliographystyle{amsplain}
\bibliography{WX}   

\end{document}